\documentclass[12pt,draftcls,onecolumn]{IEEEtran}

\usepackage{graphicx,amsmath,amssymb,cite,enumerate,setspace,epstopdf,ifpdf,psfrag,epsfig,caption,color,comment}

\ifpdf
  \DeclareGraphicsExtensions{.pdf,.jpg,.png}
\else
  \DeclareGraphicsExtensions{.eps}
\fi
\ifCLASSINFOpdf

\begin{document}

\newcommand{\qed}{\hfill \mbox{\raggedright \rule{.08in}{.086in}}}

\title{Consensus Algorithms and the Decomposition-Separation Theorem}

\author{Sadegh~Bolouki~and~%\IEEEmembership{Member,~IEEE,}
        Roland~P.~Malham\'e%\IEEEmembership{Fellow,~OSA,}
\thanks{S.~Bolouki and R.P.~Malham\'e are with the Department of Electrical Engineering, Polytechnique Montr\'eal, Montreal, QC, H3T 1J4 CA e-mail: (sadegh.bolouki,roland.malhame@polymtl.ca).}}% <-this % stops a space
%\thanks{Manuscript received April 19, 2005; revised December 27, 2012.}

\maketitle

%%%%%%%%%%%%%%%%%%%%%%%%%%%%%%%%%%%%%%%%%%%%%%%%%%%%%%%%%%%%%%%%%%%%
%%%%%%%%%%%%%%%%%%%%%%%%%%%%%%%%%%%%%%%%%%%%%%%%%%%%%%%%%%%%%%%%%%%%
%%%%%%%%%%%%%%%%%%%%%%%%%%%%%%%%%%%%%%%%%%%%%%%%%%%%%%%%%%%%%%%%%%%%

\begin{abstract}
%\boldmath
Convergence properties of time inhomogeneous Markov chain based discrete and continuous time linear consensus algorithms are analyzed. Provided that a so-called infinite jet flow property is satisfied by the underlying chains, necessary conditions for both consensus and multiple consensus are established. A recent extension by Sonin of the classical Kolmogorov-Doeblin decomposition-separation for homogeneous Markov chains to the inhomogeneous case is then employed to show that the obtained necessary conditions are also sufficient when the chain is of class $\mathcal{P}^*$, as defined by Touri and Nedi\'c. It is also shown that Sonin's theorem leads to a rediscovery and generalization of most of the existing related consensus results in the literature.

\end{abstract}

%%%%%%%%%%%%%%%%%%%%%%%%%%%%%%%%%%%%%%%%%%%%%%%%%%%%%%%%%%%%%%%%%%%%
%%%%%%%%%%%%%%%%%%%%%%%%%%%%%%%%%%%%%%%%%%%%%%%%%%%%%%%%%%%%%%%%%%%%
%%%%%%%%%%%%%%%%%%%%%%%%%%%%%%%%%%%%%%%%%%%%%%%%%%%%%%%%%%%%%%%%%%%%

\IEEEpeerreviewmaketitle

%%%%%%%%%%%%%%%%%%%%%%%%%%%%%%%%%%%%%%%%%%%%%%%%%%%%%%%%%%%%%%%%%%%%
%%%%%%%%%%%%%%%%%%%%%%%%%%%%%%%%%%%%%%%%%%%%%%%%%%%%%%%%%%%%%%%%%%%%
%%%%%%%%%%%%%%%%%%%%%%%%%%%%%%%%%%%%%%%%%%%%%%%%%%%%%%%%%%%%%%%%%%%%

\section{Introduction}

\newtheorem{theorem}{Theorem}
\newtheorem{definition}{Definition}
\newtheorem{lemma}{Lemma}
\newtheorem{corollary}{Corollary}
\newtheorem{remark}{Remark}
\newtheorem{proposition}{Proposition}
\newtheorem{example}{Example}

Linear consensus algorithms and their convergence properties have gained increasing attention in the past decade. They were first introduced in \cite{DeGroot:74}, where the author considered the case when the interactions rates between any two agents are time-invariant. Later, more general cases were considered in \cite{Chatterjee:77,Tsit:86,Tsit:84,Tsit:89a,Jadba:03,Tsit:05,Moreau:05,Hend:06,Hend:08,Li:04}. The authors aimed at identifying sufficient conditions for consensus to occur, i.e., for states to asymptotically converge to the same value. Beside consensus, multiple consensus has been the subject of many articles, e.g., \cite{Lorenz:05,Touri:10b,Touri:11c,Hend:11,Bolouki:12b}. Multiple consensus refers to the case when each agent state converges, as time grows large, to an individual limit which may or may not be different from the individual limits of other agent states. Considering the work on linear consensus algorithms, \cite{Hend:11,Touri:11c,Bolouki:12b} appear to provide the most general sufficient conditions for the occurrence of consensus or multiple consensus in a multi-agent system with dynamics described by a linear consensus algorithm.

In this paper, we deal with the limiting behavior of a general linear consensus algorithm in both discrete and continuous time. Let $\mathcal{V} = \{ 1,\ldots, N \}$ be the set of agents. In discrete time, we consider an $N$-agent system with linear update equation:
\begin{equation}
  x(t+1) = A(t) x(t), \forall t \geq 0.
  \label{DS-modelss}
\end{equation}
In (\ref{DS-modelss}), $t$ indicates the discrete time index, $x(t) = [x_1(t) \cdots x_N(t)]'$, $t \geq 0$, is the vector of agent states, where prime ($'$) indicates the transposition, $A(t)$, $t \geq 0$, is the matrix of \textit{interaction rates} $a_{ij}(t)$, $1 \leq i,j \leq N$, and $\{A(t)\}$ is the \textit{underlying chain} or \textit{transition chain} of the system, which is a chain of $(N \times N)$ row-stochastic matrices, i.e, for every $t \geq 0$, all elements of $A(t)$ are non-negative and each row of $A(t)$ sums up to 1. Throughout the paper, for simplicity, we refer to a row-stochastic matrix as a stochastic matrix. Since $A(t)$ is a stochastic matrix for every $t \geq 0$, sequence $\{ x(t) \}$, by definition, forms a \textit{backward Markov chain} with transition chain $\{A(t)\}$ (notice the evolution is described by a right hand multiplication by a column vector instead of the usual left hand multiplication by a row vector). Although we mainly focus on the discrete time case in this work, we shall extend our results to the continuous time case.

If all components of $x(t)$ asymptotically converge to the same limit, irrespective of the time index $t$ or the values at which they are initialized, \textit{unconditional global consensus}, or simply, \textit{unconditional consensus}, is said to occur. Furthermore, if there exists a fixed partition of the $N$ agents such that unconditional consensus occurs for the corresponding subvectors of $x(t)$, then \textit{unconditional multiple consensus} is said to occur. The subsets in the partition are then said to form consensus clusters. It is well known that under dynamics (\ref{DS-modelss}), unconditional consensus is equivalent to \textit{ergodicity} of chain $\{A(t)\}$ (see \cite{Chatterjee:77}), i.e., the property that backward products converge to matrices with identical rows. Furthermore, \cite{Bolouki:12b} and \cite{Touri:3} establish that a consensus algorithm with update chain $\{A(t)\}$ will induce multiple consensus if $\{A(t)\}$  is  so-called \textit{class-ergodic}, i.e., for every $t_0 \geq 0$, the product $A(t)A(t-1)\cdots A(t_0)$ converges, as $t \rightarrow \infty$. For class-ergodic chains, set $\mathcal{V}$ can be partitioned into \textit{ergodic classes}, whereby $i,j$ in $\mathcal{V}$ belong to the same ergodic class if the difference between the $i$th and $j$th rows of matrix product $A(t)A(t-1)\cdots A(t_0)$ vanishes, as $t \rightarrow \infty$. Under multiple consensus, the agent indices within the ergodic classes are the same as those within consensus clusters.

Sonin, in his so-called Decomposition-Separation (D-S) Theorem \cite{Sonin:08}, suggests an elegant and illuminating physical interpretation of the dynamics in (\ref{DS-modelss}), which we now report for completeness: Start with a forward propagating Markov chain with $(N \times N)$ transition matrices $P(t)$ and associated sequence of probability distribution vectors $m(t)$:
\begin{equation}
  m'(t+1) = m'(t) P(t), \forall t \geq 0.
\end{equation}
Interpret $m_i(t)$, $i \in \mathcal{V}$, $t \geq 0$, as the \textit{volume} of some liquid, say water for example, in a cup $i$ (out of $N$ cups), at time $t \geq 0$, while $p_{ij}(t)m_i(t)$ is the volume of liquid transferred from cup $i$ to cup $j$ at time $t \geq 0$ (see Fig. \ref{cups}).

\begin{figure}[h]
  \begin{center}
  \captionsetup{justification=centering}
    \includegraphics[scale=.6]{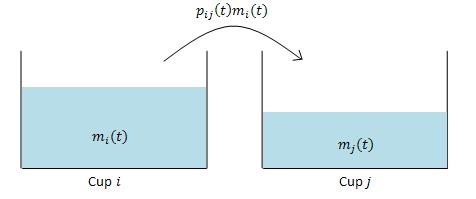}
    \caption{A physical interpretation of a Markov chain.}
    \label{cups}
  \end{center}
\end{figure}

The volume of liquid in cup $i$, $\forall i \in \mathcal{V}$, is assumed to be initialized as $m_i(0)$ at time zero. Now, let $x_i(t)$, $i \in \mathcal{V}$, $t \geq 0$, be the \textit{concentration} of a certain substance, such as sugar, alcohol, etc., within the liquid of cup $i$ at time $t$. We first assume that the volume of each cup is non-zero at all times in order to make the concentration well-defined. Moreover, assume, for every $i \in \mathcal{V}$, that $x_i(t)$ is initialized as $x_i(0)$ at time zero. It is not difficult to show that, for every $i \in \mathcal{V}$, and $t \geq 0$:
\begin{equation}
  x_i(t+1) = \frac{\sum_{j \in \mathcal{V}}p_{ji}(t)m_j(t)x_j(t)}{m_i(t+1)}.
\label{01}
\end{equation}
Let:
\begin{equation}
  x(t) \triangleq \begin{bmatrix} x_1(t) \cdots x_N(t) \end{bmatrix}',
\label{02}
\end{equation}
and $(N \times N)$ matrix $A(t)$, with elements $a_{ij}(t)$, $i,j \in \mathcal{V}$, be defined by:
\begin{equation}
  a_{ij}(t) = p_{ji}(t)m_j(t)/m_i(t+1), \forall t \geq 0.
\label{03}
\end{equation}
From (\ref{01}), (\ref{02}), and (\ref{03}), we conclude that:
\begin{equation}
  x(t+1) = A(t) x(t), \forall t \geq 0.
\end{equation}
Since $A(t)$ is stochastic for every $t \geq 0$ (check (\ref{03})), $\{ x(t) \}$ forms a backward Markov chain, with transition chain $\{ A(t) \}$, as in (\ref{DS-modelss}). Removing the non-zero volume assumption, $\{ A(t) \}$ is constructed in such a way that elements of $A(t)$, $t \geq 0$, satisfy:
\begin{equation}
  m_i(t+1) a_{ij}(t) = m_j(t)p_{ji}(t), \forall i,j \in \mathcal{V}, \forall t \geq 0.
\label{f-b}
\end{equation}
The D-S Theorem, \cite{Sonin:08}, describes the limiting behavior of both $m(t)$ and $x(t)$, as $t$ grows large. However, to take advantage of the D-S Theorem in a general consensus algorithm (\ref{DS-modelss}), one has to, first, answer the following questions: Starting with a backward propagating Markov chain generated by $\{ A(t) \}$, is it always possible to find an associated forward propagating Markov chain, with distribution vector $\{ m(t) \}$, generated by a transition chain $\{ P(t) \}$, satisfying an equation of the form (\ref{f-b})? And how, if so? As discussed in this paper, due to the existence of a so-called \textit{absolute probability sequence} for $\{ A(t) \}$, as proved in fundamental work \cite{Kolmo:36}, one could show the existence of the desired chains satisfying (\ref{f-b}). More specifically, any absolute probability sequence $\{ m(t) \}$ admitted by $\{ A(t) \}$, would help construct a forward propagating sequence of transition matrices, via (\ref{f-b}).

In this paper, it is established that, based on the D-S Theorem, all these previous results can be subsumed. Furthermore, inspired by \cite{Touri:3}, and recalling the notion of jets in Markov chains from \cite{Blackwell:45}, we introduce a property of chains resulting in necessary conditions for the unconditional occurrence of consensus or multiple consensus in (\ref{DS-modelss}). We also establish that, under an additional assumption, that is the chain being in the so-called Class $\mathcal{P}^*$ \cite{Touri:11c}, these necessary conditions also become sufficient.

In addition to the notation defined in the beginning of this section, we adopt the following notation throughout the paper. Letter $t$ stands for either discrete or continuous time indices according to context. $\Phi(t,\tau)$, $t,\tau \geq 0$, represents the state transition matrix of the considered system, which can be defined in either the discrete time domain, as in (\ref{DS-modelss}), or the continuous time domain, as we will see later on. Moreover, $\Phi_{i}(t,\tau)$ and $\Phi_{i,j}(t,\tau)$, $1 \leq i,j \leq N$, denote the $i$th column and the $(i,j)$th element (intersection of $i$th row and $j$th column) of $\Phi(t,\tau)$ respectively, while $\Phi'_i(t,\tau)$ refers to the $i$th column of $\Phi'(t,\tau)$ (the prime acts first), which is also the transpose of the $i$th row of $\Phi(t,\tau)$. For an arbitrary vector $v \in \mathbb{R}^N$, and $1 \leq i \leq N$, $v_i$ denotes the $i$th element of $v$. The overline ( $\bar{}$ ) on a subset indicates complementation of the subset in the universal set of interest.

The rest of the paper is organized as follows. In Section \ref{IJFP}, we state necessary conditions for class-ergodicity and ergodicity of a chain. The D-S Theorem, and its application in a general linear consensus algorithm, are discussed in Section \ref{D-S Theorem}. In Section \ref{Touri}, based on the D-S Theorem, we analyze the convergence properties of chains in Class $\mathcal{P}^*$. It is shown, in Section \ref{Previous Work}, that this analysis leads to a generalization of most of the existing results in the literature on convergence properties of linear consensus algorithms. A geometric approach is introduced in Section \ref{geo-approach-discrete} that applies to both discrete and continuous time consensus protocols. From the geometric framework built, we extend our analysis to the continuous time case in Section \ref{cons-cont}. Concluding remarks end the paper in Section \ref{Conclusion-DS}.

%%%%%%%%%%%%%%%%%%%%%%%%%%%%%%%%%%%%%%%%%%%%%%%%%%%%%%%%%%%%%%%%%%%%
%%%%%%%%%%%%%%%%%%%%%%%%%%%%%%%%%%%%%%%%%%%%%%%%%%%%%%%%%%%%%%%%%%%%
%%%%%%%%%%%%%%%%%%%%%%%%%%%%%%%%%%%%%%%%%%%%%%%%%%%%%%%%%%%%%%%%%%%%

\section{The Infinite Jet-Flow Property}
\label{IJFP}
Inspired by \cite{Blackwell:45}, as reported in \cite{Sonin:08} and \cite{Touri:3}, in this section, we introduce a property of chains of stochastic matrices, herein called the \textit{infinite jet-flow property}, leading to necessary conditions for ergodicity and class-ergodicity of the chain.

\begin{definition}
  For a given subset $\mathcal{V}'$ of finite set $\mathcal{V} = \{1,\ldots,N\}$, a \textit{jet} $J$ in $\mathcal{V}'$ is a sequence $\{ J(t) \}$ of subsets of $\mathcal{V}'$. A jet $J$ in $\mathcal{V}'$ is called \textit{proper} if $\emptyset \neq J(t) \subsetneq \mathcal{V}'$, $\forall t \geq 0$ (see Fig. \ref{proper jet}). Moreover, for a jet $J$, \textit{jet-limit} $J^*$ denotes the limit of the sequence $\{ J(t) \}$, as $t$ grows large, if it exists, in the sense that the sequence becomes constant after a finite time. When the elements of the sequence are all identical to a subset $S$ of $\mathcal{V}$, the jet will be referred to as jet $S$.
\label{DS-jet-def}
\end{definition}

\begin{figure}[h]
  \begin{center}
  \captionsetup{justification=centering}
    \includegraphics[scale=.5]{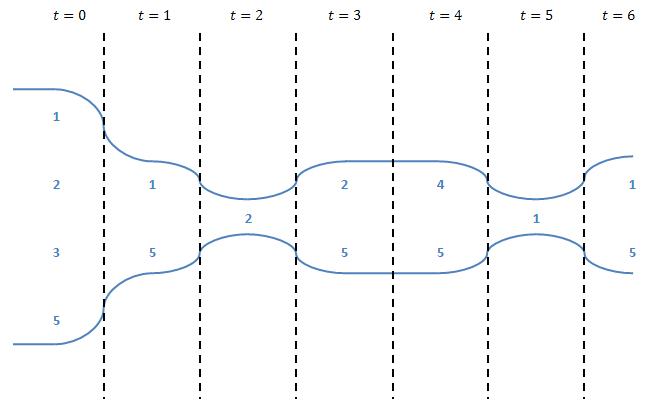}
    \caption{Example of a proper jet $J$ in $\mathcal{V}=\{1,2,3,4,5\}$: $J(0) = \{ 1,2,3,5 \}$, $J(1)=\{1,5\}$, $J(2) = \{ 2 \}$, $J(3) = \{ 2,5 \}$, $\ldots$}
    \label{proper jet}
  \end{center}
\end{figure}

\begin{definition}
  A tuple of jets $(J^1,\ldots,J^c)$ is a \textit{jet-partition} of $\mathcal{V}$, if $(J^1(t),\ldots,J^c(t))$ is a partition of $\mathcal{V}$ for every $t \geq 0$.
\end{definition}

\begin{definition}
  Let chain $\{ A(t) \}$ of stochastic matrices be given. For any two disjoint jets $J^s$ and $J^k$ in $\mathcal{V}$, $U_A(J^s,J^k)$, or simply $U(J^s,J^k)$, when no ambiguity results, denotes the total interactions between the two jets over the infinite time interval, as defined by:
	\begin{equation}
	  \begin{array}{ll}
	    \hspace{-.1in}U(J^s,J^k) = \sum_{t=0}^{\infty}& \hspace{-.1in} \Big[\sum_{i \in J^s(t+1)}\sum_{j \in J^k(t)} a_{ij}(t)\\ & + \sum_{i \in J^k(t+1)}\sum_{j \in J^s(t)} a_{ij}(t) \Big].
		\end{array}
	\label{U}
	\end{equation}
Moreover, $U_{A(t)}(J^s,J^k)$, or simply, $U_t(J^s,J^k)$, denotes the interactions between the two jets at time $t$. More specifically,
\begin{equation}
  \begin{array}{ll}
    U_t(J^s,J^k) & \hspace{-.1in}= \sum_{i \in J^s(t+1)}\sum_{j \in J^k(t)} a_{ij}(t) \vspace{.05in} \\ & \hspace{.1in}+ \sum_{i \in J^k(t+1)}\sum_{j \in J^s(t)} a_{ij}(t).
	\end{array}
\label{U_n}
\end{equation}
\label{U-def}
\end{definition}

\begin{definition}
  The \textit{complement of jet} $J$ in $\mathcal{V}$, denoted by $\mathcal{V}\backslash J$, or simply, $\bar{J}$, is the jet defined by the set sequence $\{ \mathcal{V} \backslash J(t) \}$.
\label{jet-complement-def}
\end{definition}

\begin{definition}
  A chain $\{ A(t) \}$ of stochastic matrices is said to have the \textit{infinite jet-flow property} over subset $\mathcal{V}'$ of $\mathcal{V}$ if, for every proper jet $J$ in $\mathcal{V}'$, $U(J,\mathcal{V}' \backslash J)$ is unbounded. If $\mathcal{V}' = \mathcal{V}$, chain $\{A(t)\}$ is simply said to have the infinite jet-flow property.
\label{IJFP-def}
\end{definition}

\begin{example}
  The following chain $\{A(t)\}_{t \geq 0}$ is an example of chains with the infinite jet-flow property:
\begin{equation}
  A(t) ~=~ \begin{bmatrix} 1 & 0 & 0 \\ 1-\frac{1}{t+1} & 0 & \frac{1}{t+1} \\ 0 & 0 & 1
	\end{bmatrix}, \text{ if } t \text{ is even},
\label{ex2}
\end{equation}
and
\begin{equation}
  A(t) ~=~ \begin{bmatrix} 1 & 0 & 0 \\ \frac{1}{t+1} & 0 & 1-\frac{1}{t+1} \\ 0 & 0 & 1
	\end{bmatrix}, \text{ if } t \text{ is odd}.
\label{ex1}
\end{equation}
It is not easy, at this stage, to show that chain $\{A(t)\}$ defined by (\ref{ex2}--\ref{ex1}) has the infinite jet-flow property. In Lemma \ref{weak-aperiodicity-IJFP} stated later in the paper, we suggest a way to check the infinite jet-flow property of a chain that implies the infinite jet-flow property of $\{A(t)\}$ defined by (\ref{ex2}--\ref{ex1}).
\label{example with IJFP}
\end{example}

\begin{example}
Chain $\{A(t)\}_{t \geq 0}$ defined by:
\begin{equation}
  A(t) ~=~ \begin{bmatrix} 1 & 0 & 0 \\ 1-\frac{1}{(t+1)^2} & 0 & \frac{1}{(t+1)^2} \\ 0 & 0 & 1
	\end{bmatrix}, \text{ if } t \text{ is even},
\label{ex4}
\end{equation}
and
\begin{equation}
  A(t) ~=~ \begin{bmatrix} 1 & 0 & 0 \\ \frac{1}{(t+1)^2} & 0 & 1-\frac{1}{(t+1)^2} \\ 0 & 0 & 1
	\end{bmatrix}, \text{ if } t \text{ is odd},
\label{ex3}
\end{equation}
is an example of chains for which the infinite jet-flow property is not satisfied. More specifically, if we define jet $J$ by:
\begin{equation}
	J(t) ~=~
	\begin{cases}
		\{ 1 \}	&	\text{ if $t$ is even}\\
		\{ 1 , 2 \}	&	\text{ if $t$ is odd}
	\end{cases}
\end{equation}
then we have:
\begin{equation}
	U(J, \mathcal{V} \backslash J) ~=~ \sum_{t=0}^{\infty} \frac{1}{(t+1)^2} ~<~ \infty,
\end{equation}
which shows that the infinite jet-flow property does not hold.
\label{example without IJFP}
\end{example}

In the following proposition, we state a sufficient condition for the infinite jet-flow property to hold.

\begin{definition}\cite{Touri:11c}
\label{infinite flow graph-1}
  For a chain $\{ A(t) \}$ of stochastic matrices, we define its \textit{infinite flow graph}, $G_A(\mathcal{V},E)$, by an undirected graph of size $N$, such that:
	\begin{equation}
	  E = \{ (i,j)|i,j \in \mathcal{V},\, i \neq j,\, \sum_{t=0}^\infty (a_{ij}(t)+a_{ji}(t)) = \infty \}.
	\end{equation}
The set of nodes of each connected component of $G_A(\mathcal{V},E)$ is called an \textit{island} of $\{A(t)\}$. Moreover, chain $\{A(t)\}$ is said to have the \textit{infinite flow property} if and only if $G_A(\mathcal{V},E)$ is connected.
\label{island-def}
\end{definition}

The following theorem states a necessary condition for class-ergodicity of chain $\{A(t)\}$ of stochastic matrices.

\begin{theorem}
  A chain $\{ A(t) \}$ of stochastic matrices is class-ergodic only if the infinite jet-flow property holds over each island of $\{A(t)\}$.
\label{IJFP-over-islands}
\end{theorem}

\begin{proof}
  Assume that, on the contrary, $\{ A(t) \}$ is class-ergodic, yet some proper jet $J$, in an island $I$ of $\{ A(t) \}$, is such that $U_A(J, I \backslash J)$ is bounded. Recall, from Definition \ref{DS-jet-def}, that by a proper jet in $I$, we mean $\emptyset \neq J(t) \subsetneq I$, $\forall t \geq 0$. Since $U_A(J, I \backslash J)$ is bounded and $I$ is an island of $\{ A(t) \}$, we conclude that $U_A(J, \mathcal{V} \backslash J)$ is bounded as well. Recalling the definition of $l_1$-approximation from \cite{Touri:10b}, a chain $\{B(t)\}$ is an $l_1$-approximation of chain $\{A(t)\}$ if:
  \begin{equation}
    \sum_{t=0}^{\infty} \| A(t)-B(t) \| < \infty,
  \end{equation}
  where for convenience only, the norm refers to the \textit{max norm}, i.e., the maximum of the absolute values of the matrix elements. We now form chain $\{ B(t) \}$, an $l_1$-approximation of chain $\{ A(t) \}$, by eliminating interactions between $J$ and $\mathcal{V}\backslash J$ at all times. From \cite[Lemma 1]{Touri:10b}, it is known that $l_1$-approximations do not influence the ergodic classes of a chain. Therefore, $\{ B(t) \}$ will remain class-ergodic with the same ergodic classes as $\{A(t)\}$. Also, the islands of B(t) are the same as those of A(t). On the other hand, $U_B(J, \mathcal{V} \backslash J)=0$. Given two distinct arbitrary constants, $\alpha_1$ and $\alpha_2$, let states of a multi-agent system, $y_i(t)$, $i \in \mathcal{V}$, evolve via dynamics $y(t+1) = B(t) y(t)$, $\forall t \geq 0$, and be initialized at: $y_i(0) = \alpha_1$ if $i \in J(0)$, and $y_i(0) = \alpha_2$ otherwise. Since there is no interaction between $J$ and $\mathcal{V} \backslash J$ at any time, we conclude that for every $t \geq 0$, we have: $y_i(t) = \alpha_1$ if $i \in J(t)$, and $y_i(t) = \alpha_2$ otherwise. Since $\{ B(t) \}$ is class-ergodic, $\lim_{t \rightarrow \infty}y_i(t)$ exists for every $i \in \mathcal{V}$ and the consensual agents can be grouped into clusters sharing the same limit and forming an ergodic class. Since the elements in $\{J(t)\}$ are always associated with the same value of $y$ for any $t$, they will asymptotically belong to a fixed limiting cluster $S^*$ , namely agents for which $y_i(t)$ converges to $\alpha_1$. Since $J$ is a proper jet in $I$, we have: $\emptyset \neq S^* \subsetneq I$. Consider, now, jet $S^*$ on island $I$. $S^*$ is essentially the limiting jet $J^*$ of $J$. Since the island structure is common for chains $\{A(t)\}$ and $\{B(t)\}$, we know that $U_B(J^*,I \backslash J^*)$ is unbounded. This is in contradiction with $U_B(J, I \backslash J) \leq U_B(J, \mathcal{V} \backslash J)=0$, which completes the proof.
\end{proof}

Later in this paper, we shall establish the sufficiency of the infinite jet-flow property in Theorem \ref{IJFP-over-islands}, provided $\{ A(t) \}$ is in Class $\mathcal{P}^*$, as defined in \cite{Touri:11c}. We now note that the infinite flow property of $\{ A(t) \}$, which is a necessary condition for ergodicity of $\{ A(t) \}$ according to \cite{Touri:10a}, is equivalent to the existence of a single island. Thus, Theorem \ref{IJFP-over-islands} immediately results in the following corollary which is a necessary condition for ergodicity of chain $\{A(t)\}$ of stochastic matrices.

\begin{corollary}
  A chain $\{ A(t) \}$ of stochastic matrices is ergodic only if it has the infinite jet-flow property.
\label{IJFP-necessary}
\end{corollary}

Corollary \ref{IJFP-necessary} provides a more restrictive necessary condition for ergodicity of a chain than Theorems 1 and 2 of \cite{Touri:3}. For instance, from Corollary \ref{IJFP-necessary}, we conclude that the chain of Example \ref{example without IJFP} is not ergodic since it does not have the infinite jet-flow property. However, this cannot be concluded from Theorem 1 and 2 of \cite{Touri:3}.

On the other hand, we notice that the infinite jet-flow property is not sufficient for ergodicity. For instance, one can verify that the chain of Example \ref{example with IJFP} is not ergodic while the infinite jet-flow property holds.

\begin{definition}
  A jet $J$ in $\mathcal{V}$ is called an \textit{independent jet} if the total influence of $\bar{J}$ on $J$ is finite over the infinite time interval, i.e.,
	\begin{equation}
	  \sum_{t=0}^{\infty}\sum_{i \in J(t+1)}\sum_{j \in \bar{J}(t)} a_{ij}(t) < \infty.
	\end{equation}
\label{leader-def}
\end{definition}

The following theorem, which is a generalization of Corollary \ref{IJFP-necessary}, states yet another necessary condition for ergodicity of chain $\{A(t)\}$ of stochastic matrices.

\begin{theorem}
  A chain $\{ A(t) \}$ of stochastic matrices is ergodic only if no two disjoint independent jets in $\mathcal{V}$ exist.
\label{no-two-leaders-necessary}
\end{theorem}

\begin{proof}
  Assume that on the contrary, there exist two disjoint independent jets $J^1$ and $J^2$ in $\mathcal{V}$. Similar to the proof of Theorem \ref{IJFP-over-islands}, form chain $\{ B(t) \}$, an $l_1$-approximation of $\{ A(t) \}$, by eliminating the influence of $\bar{J}^s$ on $J^s$, $s=1,2$, at all times. Recall that $\{ A(t) \}$ and $\{ B(t) \}$ will share the same ergodicity properties. Let states of a multi-agent system, $y_i(t)$, $1 \leq i \leq N$, evolve via dynamics $y(t+1) = B(t) y(t)$, $\forall t \geq 0$, and be initialized such that for every $i \in J^s(0)$ ($s=1,2$), $y_i(0) = \alpha_s$, where $\alpha_1 \neq \alpha_2$. Then, for every $t \geq 0$, we have: $y_i(t) = \alpha_s$, $\forall i \in J^s(t)$ ($s=1,2$). Since $\alpha_1 \neq \alpha_2$, consensus does not occur. Consequently, chain $\{ B(t) \}$ and thus $\{ A(t) \}$ could not possibly be ergodic.
\end{proof}

As an example, for chain $\{ A(t) \}$ of Example \ref{example with IJFP}, jet $\{ 1 \}$ and jet $\{ 3 \}$ are two disjoint independent jets in $\mathcal{V}=\{ 1,2,3 \}$. Thus, Theorem \ref{no-two-leaders-necessary} implies that $\{ A(t) \}$ is not ergodic.

\begin{remark}
	The following argument explains why Theorem \ref{no-two-leaders-necessary} generalizes Corollary \ref{IJFP-necessary}. Without the infinite jet-flow property, there exists a jet $J$ such that $U(J,\mathcal{V}' \backslash J)$ is bounded. Thus, both jets $J$ and $\mathcal{V} \backslash J$ are independent jets. On the other hand, jet $J$ and $\mathcal{V} \backslash J$ are disjoint. Thus, infinite jet-flow is a weaker condition than the non-existence of any two disjoint independent jets.
\end{remark}
%%%%%%%%%%%%%%%%%%%%%%%%%%%%%%%%%%%%%%%%%%%%%%%%%%%%%%%%%%%%%%%%%%%%
%%%%%%%%%%%%%%%%%%%%%%%%%%%%%%%%%%%%%%%%%%%%%%%%%%%%%%%%%%%%%%%%%%%%
%%%%%%%%%%%%%%%%%%%%%%%%%%%%%%%%%%%%%%%%%%%%%%%%%%%%%%%%%%%%%%%%%%%%

\section{Relationship to the D-S Theorem}
\label{D-S Theorem}
Consider a multi-agent system with states evolving according to linear algorithm (\ref{DS-modelss}), where $\{ A(t) \}$ is a chain of stochastic matrices. Based on the work of Kolmogorov in \cite{Kolmo:36}, we know that for every chain $\{ A(t) \}_{t \geq 0}$, there exists a sequence $\{\pi(t)\}_{t \geq 0}$ of probability distribution vectors, called \textit{absolute probability sequence}, such that
\begin{equation}
  \pi'(t+1)A(t) = \pi'(t), \forall t \geq 0.
\end{equation}
The transition chain $\{ P(t) \}$ of the forward propagating chain associated with $\{ A(t) \}$ and $\{ \pi(t) \}$ as in (\ref{f-b}), must be such that:
\begin{equation}
  \pi_i(t) p_{ij}(t) = \pi_j(t+1) a_{ji}(t), \forall i,j \in \mathcal{V}, \forall t \geq 0.
\end{equation}
More specifically, if $\pi_i(t) \neq 0$, then:
\begin{equation}
  p_{ij}(t) = \pi_j(t+1) a_{ji}(t) / \pi_i(t),
\end{equation}
while if $\pi_i(t) = 0$, for some $i$ and $t \geq 0$, we choose $p_{ij}(t)$'s non-negative, arbitrarily such that:
\begin{equation}
  \sum_{j=1}^N p_{ij}(t) = 1.
\label{DS-sto}
\end{equation}
Note that in the former case ($\pi_i(t) \neq 0$), (\ref{DS-sto}) is automatically satisfied, implying that $P(t)$ is a stochastic matrix for every $t \geq 0$. It is easy to see that:
\begin{equation}
  \pi'(t) P(t) = \pi'(t+1), \forall t \geq 0.
\end{equation}
Thus, $\{ \pi(t) \}$ forms the probability distribution vector of an inhomogeneous forward propagating Markov chain. Let $V(J^s,J^k)$ denote the total \textit{flow} between two arbitrary jets $J^s$ and $J^k$ in $\mathcal{V}$ over the infinite time interval as defined by:
\begin{equation}
  \begin{array}{ll}
    \hspace{-.13in} V(J^s,J^k) = \sum_{t=0}^{\infty} & \hspace{-.1in} \Big[\sum_{i \in J^k(t)}\sum_{j \in J^s(t+1)} r_{ij}(t)\\ & + \sum_{i \in J^s(t)}\sum_{j \in J^k(t+1)} r_{ij}(t) \Big],
	\end{array}
\label{V}
\end{equation}
where
\begin{equation}
  r_{ij}(t) = \pi_i(t)p_{ij}(t) = \pi_j(t+1)a_{ji}(t).
\end{equation}
Value $r_{ij}(t)$ can be interpreted as the absolute joint probability of being in $i$ at time $t$ and $j$ at time $t+1$. Recalling $U$ from (\ref{U}), we note that for every $J^s,J^k$ in $\mathcal{V}$, $V(J^s,J^k) \leq U(J^s,J^k)$. Sonin, in his elegant work \cite{Sonin:08}, characterizes the limiting behavior of the two sequences $\{ \pi(t) \}$ and $\{ x(t) \}$ (evolving via (\ref{DS-modelss})) in the so-called D-S Theorem as the following.
\begin{theorem} \textit{(Sonin's D-S Theorem)}
There exists an integer $c$, $1 \leq c \leq N$, and a decomposition of $\mathcal{V}$ into jet-partition $(J^0,J^1,\ldots,J^c)$, $J^k = \{ J^k(t) \}$, $0 \leq k \leq c$, such that irrespective of the particular time or state at which $x_i$'s are initialized,
\begin{enumerate}[(i)]
  \item For every $k$, $1 \leq k \leq c$, there exist constants $\pi^*_k$ and $x^*_k$, such that
	\begin{equation}
	  \lim_{t \rightarrow \infty} \sum_{i \in J^k(t)} \pi_i(t)=\pi^*_k,
	\end{equation}
	and
	\begin{equation}
	  \lim_{t \rightarrow \infty} x_{i_t}(t)=x^*_k,
	\end{equation}
	for every sequence $\{ i_t \}$, $i_t \in J^k(t)$. Furthermore, $\lim_{t \rightarrow \infty}\sum_{i \in J^0(t)} \pi_i(t) = 0$.
	\item For every distinct $k,s$, $0 \leq k,s \leq c$: $V(J^k,J^s) < \infty$.
	\item This decomposition is unique up to jets $\{ J(t) \}$ such that for any $\{ \pi(t) \}$ we have:
	\begin{equation}
	  \lim_{t \rightarrow \infty}\sum_{i \in J(t)} \pi_i(t) = 0 \text{ and } V(J,\mathcal{V}\backslash J) < \infty.
	\end{equation}
\end{enumerate}
\end{theorem}
We shall take advantage of the Sonin's D-S Theorem to characterize the asymptotic behavior of a class of chains of stochastic matrices in the following section.
%On the other hand, we have:
%\begin{equation}
%  r_{ij}(n) = \pi_i(n)p_{ij}(n) = \pi_j(n+1)a_{ji}(n).
%\label{r}
%\end{equation}
%Relations (\ref{V}) and (\ref{r}) imply that:
%\begin{equation}
%  V(J^s,J^k) = \sum_{n=0}^{\infty}\left[\sum_{i \in J^s(n+1)}\sum_{j \in J^k(n)} \left(\pi_i(n+1)a_{ij}(n)\right) + \sum_{i \in J^k(n+1)}\sum_{j \in J^s(n)} \left(\pi_i(n+1)a_{ij}(n)\right) \right]
%\label{weighted-flow}
%\end{equation}

%%%%%%%%%%%%%%%%%%%%%%%%%%%%%%%%%%%%%%%%%%%%%%%%%%%%%%%%%%%%%%%%%%%%
%%%%%%%%%%%%%%%%%%%%%%%%%%%%%%%%%%%%%%%%%%%%%%%%%%%%%%%%%%%%%%%%%%%%
%%%%%%%%%%%%%%%%%%%%%%%%%%%%%%%%%%%%%%%%%%%%%%%%%%%%%%%%%%%%%%%%%%%%

\section{Convergence in Class $\mathcal{P}^*$}
\label{Touri}

In this section, we apply Sonin's D-S Theorem to chains in class $\mathcal{P}^*$ as first defined in \cite{Touri:11c}.

\begin{definition}\cite[Definition 3]{Touri:11c}
	Chain $\{ A(t) \}$ is said to be in class $\mathcal{P}^*$ if it admits an absolute probability sequence uniformly bounded away from zero, i.e., there exists $p^* > 0$ such that
\begin{equation}
  \pi_i(t) \geq p^*, \forall i \in \mathcal{V}, \forall t \geq 0.
\end{equation}
\end{definition}

For chains in Class $\mathcal{P}^*$, it is immediately implied that in the jet decomposition of the D-S Theorem, there is no jet $J^0$. Otherwise, $\lim_{t \rightarrow \infty}\sum_{i \in J^0(t)} \pi_i(t)$ would be bounded away from zero by at least $p^*$, which is in contradiction with the D-S Theorem. Therefore, there is a jet-partition of $\mathcal{V}$ into jets $J^1,\ldots,J^c$, such that for every $k=1,\ldots,c$, $\lim_{t \rightarrow \infty} x_{i_t}(t) =x^*_k$, for every sequence $\{ i_t \}$, where $i_t \in J^k(t)$. Thus, we have the following proposition for chains in Class $\mathcal{P}^*$.
\begin{proposition}
  Consider a multi-agent system with dynamics (\ref{DS-modelss}), where chain $\{ A(t) \}$ is in Class $\mathcal{P}^*$. Then, the set of accumulation points of states is finite.
\label{acc-points-finite}
\end{proposition}
\begin{proof}
  Obvious if we note that $\{ x^*_k | 1 \leq k \leq c \}$ form the set of accumulation points of states.
\end{proof}
\begin{lemma}
  If $\{ A(t) \} \in \mathcal{P}^*$, then for every two jets $J^1$ and $J^2$ in $\mathcal{V}$, $V(J^1,J^2)=\infty$ if and only if $U(J^1,J^2)=\infty$.
\label{infinite V-U}
\end{lemma}
\begin{proof}
  The result is obvious if one notes that
	\begin{equation}
	  p^* U(J^1,J^2) \leq V(J^1,J^2) \leq U(J^1,J^2).
	\end{equation}
\end{proof}

\begin{theorem}
  A chain $\{ A(t) \}$ in Class $\mathcal{P}^*$ is class-ergodic if and only if the infinite jet-flow property holds over each island of $\{ A(t) \}$. In case of class-ergodicity of $\{ A(t)\}$, islands are the ergodic classes of $\{ A(t) \}$, and constitute the jet limits in the jet decomposition of $\{A(t)\}$. Moreover, these limits are attained in finite time.
\label{IJFP-equivalent}
\end{theorem}

\begin{proof}
  We first assume that chain $\{A(t)\}$ in $\mathcal{P}^*$ is class-ergodic. Then, Theorem \ref{IJFP-over-islands} implies that the infinite jet-flow property holds over each island of the chain. We now show that if $\{ A(t) \} \in \mathcal{P}^*$ is class-ergodic, islands are the ergodic classes of $\{ A(t) \}$. Let us call an agent $i \in \mathcal{V}$, a \textit{prime member} of jet $J^k$ if $i \in J^k(t)$ for infinitely many times. Having defined the prime membership, there exists some Sonin's jet-decomposition of $\{ A(t) \}$ such that each agent becomes the prime member of a unique jet. To obtain such a jet-decomposition, start with an arbitrary jet-decomposition and let any two jets with a common prime member merge. The merging process results in a Sonin's jet-decomposition with the desired property. Jets of such decomposition have the property that they become time-invariant after a finite time. Thus, the jet-limits exist for each jet and are ergodicity classes of $\{ A(t) \}$. If $i$ and $j$ belong to the same jet-limit, they are in the same island since they are in the same ergodic class of $\{ A(t) \}$ (\cite{Touri:10b}, Lemma 2). Conversely, assume that $i$ and $j$ are neighbors in the infinite flow graph, i.e., $\sum_{t=0}^{\infty} (a_{ij}(t)+a_{ji}(t))=\infty$. If $i$ and $j$ were to belong to different jet-limits $J^{s^*},J^{k^*}$, then $U(J^s,J^k)$ would be unbounded. Thus, based on Lemma \ref{infinite V-U}, $V(J^s,J^k)$ would be unbounded as well, which contradicts property (ii) in the D-S theorem. Therefore, every two neighbors in the infinite flow graph belong to the same jet-limit. Consequently, every $i$ and $j$ in the same island must be in the same jet-limit.% Since islands are the jet-limits, uniqueness of the jet decomposition implies that the infinite jet-flow property must hold over each island.
  
	To prove the sufficiency, we assume that the infinite jet-flow property holds over each island. Let $(J^1,\ldots,J^c)$ be a Sonin's jet-decomposition, and for every $k=1,\ldots,c$, $\lim_{t \rightarrow \infty} x_{i_t}(t) =x^*_k$ for every sequence $\{ i_t \}$, where $i_t \in J^k(t)$.	Let $I$ be an arbitrary island. We aim to show that, for every $i \in I$, $\lim_{t \rightarrow \infty}x_i(t)$ exists. To this aim, keeping in mind that the aim is achieved is one of jets $J^1,\ldots,J^c$ contains island $I$ after some finite time, we follow three steps. Pick an arbitrary jet $J^k$ among $J^1,\ldots,J^c$.
	
	\textit{Step 1:} We show that, infinitely often, we have: $I \cap J^k(t) = \emptyset$ or $I \cap J^k(t) = I$, where $\emptyset$ denotes the empty set. Indeed, assume instead that this behavior occurs only a finite number $r$ of times, denoted $t_1,\ldots,t_r$. We form a proper jet $J$ in $I$ such that:
	\begin{equation}
	  J(t) = I \cap J^k(t), \text{ if } t \neq t_i, 1 \leq i \leq r.
	\end{equation}
	Since the infinite jet-flow property holds over $I$, $U(J,I \backslash J)$ is unbounded. On the other hand, except for a finite number of time indices $t = t_i$, $1 \leq i \leq r$, $U_t(J,I \backslash J)\leq U_t(J^k, \mathcal{V} \backslash J^k)$. This implies that $U(J^k,\mathcal{V} \backslash J^k)$ is unbounded, and, according to Lemma \ref{infinite V-U}, so is $V(J^k, \mathcal{V} \backslash J^k)$. This is in contradiction with the D-S Theorem. Therefore, $I \cap J^k(t) = \emptyset$ or $I$ happens infinitely many times. This means that either one or both of the events $I \cap J^k(t) = \emptyset$ and $I \cap J^k(t) = I$ occurs infinitely often.
%\begin{lemma}
%  The following two statements hold:
%	\begin{enumerate}[(i)]
%		\item There are finite times $n$ such that $I \subset J^k(n)$ and $I \not\subset J^k(n+1)$.
%		\item There are finite times $n$ such that $I \not\subset J^k(n)$ and $I \subset J^k(n+1)$.
%	\end{enumerate}
%\end{lemma}
	  
	  \textit{Step 2:} We show that there are at most a finite number of times such that $I \subseteq J^k(t)$ and $I \not\subseteq J^k(t+1)$. Indeed, denote:
  \begin{equation}
  	\epsilon \triangleq \frac{1}{3}\min\{ |x^*_s-x^*_l|\, |1 \leq s \neq l \leq c \},
	\label{epsilon}
  \end{equation}
  there exists $T_{\epsilon}\geq 0$ such that:
  \begin{equation}
    |x_i(t) - x^*_l| < \epsilon, \forall l=1,\ldots,c, \forall i \in J^l(t), \forall t \geq T_{\epsilon}.
	\label{T_eps}
  \end{equation}
  For some given $t \geq T_{\epsilon}$ assume that: $I \subseteq J^k(t)$ and $I \not\subseteq J^k(t+1)$. Then, there exists $i \in I$ such that $i \in J^k(t) \backslash J^k(t+1)$. In view of (\ref{DS-modelss}), (\ref{epsilon}), and (\ref{T_eps}), we then have:
  \begin{equation}
    |\sum_{j \not\in J^k(t)} a_{ij}(t)(x_j(t)-x_i(t))| \geq \epsilon.
  \label{1110}
  \end{equation}
  On the other hand,
  \begin{equation}
	  \begin{array}{l}
      \hspace{-.5in}|\sum_{j\not\in J^k(t)} a_{ij}(t)(x_j(t)-x_i(t))|\vspace{.05in} \\ \hspace{.2in} \leq \sum_{j \not\in J^k(t)} a_{ij}(t) |x_j(t)-x_i(t)|\vspace{.05in} \\ \hspace{.2in} \leq L \sum_{j \not\in J^k(t)} a_{ij}(t),
	  \end{array}
  \label{1120}	
  \end{equation}
  where
  \begin{equation}
    L \triangleq \max\{ x_j(0)-x_i(0), | i,j \in \mathcal{V}\}.
	\label{DS-L}
  \end{equation}
	Note that $L$ remains an upper bound of $| x_j(t)-x_i(t) |$, $\forall t \geq 0$, since states are updated via a convex combination of previous states. Eqs. (\ref{1110}) and (\ref{1120}) imply:
  \begin{equation}
    \sum_{j \not\in J^k(t)} a_{ij}(t) \geq \epsilon / L.
  \end{equation}
	Therefore, since $i \in I$:
  \begin{equation}
    \sum_{l \in I}\sum_{j \not\in I} a_{lj}(t) \geq \sum_{j \not\in {I}} a_{ij}(t) \geq \sum_{j \not\in J^k(t)} a_{ij}(t) \geq \epsilon / L.
	\label{1130}
  \end{equation}
	Since $U(I, \mathcal{V} \backslash I)< \infty$, inequality (\ref{1130}) can only occur for finitely many times $t$.	This shows that if $I \subseteq J^k(t)$ happens infinite times, then there exists $T$ such that $I \subseteq J^k(t)$ for every $t \geq T$. Consequently, $\lim_{t \rightarrow \infty} x_i(t)$ exists, $\forall i \in I$, and is equal to $x^*_k$. Therefore, assume that for a fixed island $I$, $I \subseteq J^k(t)$ happens only a finite number of times for every $k$, $1 \leq k \leq c$. Thus, from the result of \textit{Step 1}, $I \cap J^k(t) = \emptyset$ must happen infinite times, for every $k$, $1 \leq k \leq c$.
	
	\textit{Step 3:} We show that if $I \cap J^k(t) = \emptyset$ happens infinite times, for every $k$, $1 \leq k \leq c$, then, the following contradiction occurs: For every $k$, $1 \leq k \leq c$, there exists $T_k \geq 0$ such that $I \cap J^k(t) = \emptyset$, $\forall t \geq T_k$. The proof is established by induction on $k$. With no loss of generality, assume that $x^*_1 < \cdots < x^*_k$. $(k=1)$: Recalling $\epsilon$ and $T_{\epsilon}$ from (\ref{epsilon}) and (\ref{T_eps}), assume that for a fixed $t \geq T_{\epsilon}$ we have $I \cap J^1(t) = \emptyset$ and $I \cap J^1(t+1) \neq \emptyset$. Thus, there exists $i \in I$ such that $i \in J^1(t+1) \backslash J^1(t)$. Therefore,
	\begin{equation}
	  \sum_{j \in J^1(t)}|a_{ij}(t) (x_j(t)-x_i(t))| \geq \epsilon.
	\end{equation}
	Noting that $J^1(t) \subseteq \mathcal{V}\backslash I$, by repeating steps (\ref{1110})-(\ref{1130}), we conclude that there are at most finitely many times at which $I \cap J^1(t) = \emptyset$ and $I \cap J^1(t+1) \neq \emptyset$. This together with the fact that $I \cap J^1(t) = \emptyset$ happens infinite times, shows that there exists $T_1 \geq 0$ such that $I \cap J^1(t) = \emptyset$, $\forall t \geq T_1$.\\ $k-1 \rightarrow k$ ($1 < k \leq c$): Assume that for a fixed $t \geq \max\{T_{l}|1 \leq l < k\}$, we have $I \cap J^{k}(t) = \emptyset$ and $I \cap J^k(t+1) \neq \emptyset$. Thus, there exists $i \in I$ such that $i \in J^k(t+1) \backslash J^k(t)$. Therefore,
	\begin{equation}
	  \sum_{j \in \bigcup_{l=1}^{k}J^l(t)}|a_{ij}(t) (x_j(t)-x_i(t))| \geq \epsilon.
	\end{equation}
	Once again, we note that $\bigcup_{l=1}^{k}J^l(t) \subseteq \bar{I}$, and repeat steps (\ref{1110})-(\ref{1130}) to show that there exists $T_k \geq 0$ such that $I \cap J^k(t) = \emptyset$, $\forall t \geq T_k$.
\end{proof}

\begin{corollary}
  A chain $\{ A(t) \} \in \mathcal{P}^*$ is ergodic if and only if it has the infinite jet-flow property.
\label{IJFP-ergodicity}
\end{corollary}

Since convergence of states occurs inside each jet $J^k$, $1 \leq k \leq c$, for multiple consensus to occur unconditionally (class-ergodicity of $\{ A(t) \}$), it suffices that for each jet of the D-S Theorem jet decomposition, its jet-limit exists.

%%%%%%%%%%%%%%%%%%%%%%%%%%%%%%%%%%%%%%%%%%%%%%%%%%%%%%%%%%%%%%%%%%%%
%%%%%%%%%%%%%%%%%%%%%%%%%%%%%%%%%%%%%%%%%%%%%%%%%%%%%%%%%%%%%%%%%%%%
%%%%%%%%%%%%%%%%%%%%%%%%%%%%%%%%%%%%%%%%%%%%%%%%%%%%%%%%%%%%%%%%%%%%

\section{Relationship to Previous Work}
\label{Previous Work}

%%%%%%%%%%%%%%%%%%%%%%%%%%%%%%%%%%%%%%%%%%%%%%%%%%%%%%%%%%%%%%%%%%%%
%%%%%%%%%%%%%%%%%%%%%%%%%%%%%%%%%%%%%%%%%%%%%%%%%%%%%%%%%%%%%%%%%%%%

\subsection{Weakly Aperiodic Chains in Class $\mathcal{P}^*$}
\label{weak-aperiodicity}
In this section of the paper, we see how the weak aperiodicity property, as defined in \cite{Touri:11c}, guarantees that the infinite jet-flow property holds over each island. In accordance with \cite{Touri:11c}, weak aperiodicity of a chain is defined as follows:
\begin{definition}
  A chain $\{ A(t) \}$ of stochastic matrices is said to be \textit{weakly aperiodic} if there exists $\gamma > 0$ such that for every distinct $i,j \in \mathcal{V}$ and each $t \geq 0$, there exists $l \in \mathcal{V}$ such that
	\begin{equation}
	  a_{li}(t).a_{lj}(t) \geq \gamma a_{ij}(t).
	  \label{w-ap}
	\end{equation}
\end{definition}
\begin{lemma}
  Let $\{ A(t) \}$ be a chain of stochastic matrices in Class $\mathcal{P}^*$ that is weakly aperiodic. Then, the infinite jet-flow property holds over each island of $\{ A(t) \}$. In particular, in presence of a single island, the infinite jet-flow property holds for chain $\{ A(t) \}$.
\label{weak-aperiodicity-IJFP}
\end{lemma}
\begin{proof}
Let $\{ A(t) \}$ be weakly aperiodic, $I$ be an arbitrary island of $\{ A(t) \}$, and $J$ be an arbitrary jet in $I$. If jet-limit $J^*$ exists, since $I$ is a connected component of the infinite flow graph, $U(J^*,I\backslash J^*)$ is unbounded. Consequently, $U(J,I\backslash J)$ is unbounded and the lemma holds. Thus instead, assume that for jet $J$, the jet-limit does not exist. Therefore, for infinitely many times $t$, we must have: $J(t+1) \not\subseteq J(t)$. Let $t$ be fixed and $J(t+1) \not\subseteq J(t)$. Thus, there exists $i \in J(t+1) \backslash J(t)$. From the weak aperiodicity property of $\{ A(t) \}$ (see (\ref{w-ap})), for every $j \in J(t)$, there exists $l \in \mathcal{V}$ such that:
\begin{equation}
  \begin{array}{ll}
    \gamma a_{ij}(t) \leq a_{li}(t).a_{lj}(t) \vspace{.05in} & \hspace{-.1in} \leq \min\{ a_{li}(t),a_{lj}(t) \} \\ & \hspace{-.1in} \leq U_t(J,\mathcal{V} \backslash J),
	\end{array}
\end{equation}
where $U_t$ is defined in (\ref{U_n}). The reason for the last inequality is that, whether $l \in J(t+1)$ or $l \not\in J(t+1)$, one of $a_{li}(t),a_{lj}(t)$ appears in $U_t(J, \mathcal{V} \backslash J)$. Hence,
\begin{equation}
  \sum_{j \in J(t)} \gamma a_{ij}(t) \leq |J(t)| U_t(J,\mathcal{V} \backslash J).
\label{2100}
\end{equation}
On the other hand,
\begin{equation}
  \begin{array}{ll}
    \sum_{j \in J(t)} \gamma a_{ij}(t) & \hspace{-.1in} = \gamma \sum_{j \in J(t)} a_{ij}(t) \vspace{.05in}\\ & \hspace{-.1in} = \gamma \left( 1 - \sum_{j \not\in J(t)} a_{ij}(t) \right) \vspace{.05in} \\ & \hspace{-.1in} \geq \gamma \left(1 - U_t(J,\mathcal{V} \backslash J)\right).
	\end{array}
\label{2110}
\end{equation}
Relations (\ref{2100}) and (\ref{2110}) imply:
\begin{equation}
  U_t(J,\mathcal{V} \backslash J) \geq \gamma /(\gamma + |J(t)|) > \gamma / (\gamma+N).
\label{2120}
\end{equation}
Since (\ref{2120}) holds for infinitely many times $t$, $U(J,\mathcal{V} \backslash J)=\sum_{t=0}^{\infty}U_t(J,\mathcal{V} \backslash J)$ is unbounded, and so is $U(J,I \backslash J)$ (since $J$ is a jet in $I$, and $I$ is an island).
\end{proof}

Theorem \ref{IJFP-equivalent} and Lemma \ref{weak-aperiodicity-IJFP} immediately imply the following corollary which is the deterministic counterpart of Theorem 4 of \cite{Touri:11c}.

\begin{corollary}
  Every weakly aperiodic chain in Class $\mathcal{P}^*$ is class-ergodic.
\end{corollary}

Note that an equivalent definition of weak periodicity is as follows.
\begin{definition}
  A chain $\{ A(t) \}$ of stochastic matrices is \textit{weakly aperiodic} if there exists $\gamma > 0$ such that for every distinct $i,j \in \mathcal{V}$ and each $t \geq 0$, there exists $l \in \mathcal{V}$ such that
	\begin{equation}
	  \min\{ a_{li}(t),a_{lj}(t) \} \geq \gamma a_{ij}(t).
	\end{equation}
\label{weakly-aperiodic-def}
\end{definition}

To achieve class-ergodicity under the $\mathcal{P}^*$ class assumption, the number of times in which an agent moves from a jet to another must be finite. Indeed, let
\begin{equation}
	\epsilon \triangleq \frac{1}{3}\min\{ |x^*_s-x^*_k|\, |1 \leq k \neq s \leq c \}.
\end{equation}
Then, there exists $T_{\epsilon}$ such that for every $t \geq T_{\epsilon}$,
\begin{equation}
  |x_i(t) - x^*_k| < \epsilon, \forall i \in J^k(t).
\end{equation}
If agent $i$ moves from a jet, say $J^1$, to another jet, say $J^2$, at time $t$ ($i \in J^1(t) \cap J^2(t+1)$), we must have:
\begin{equation}
  |\sum_{j\not\in J^1(t)} a_{ij}(t)(x_j(t)-x_i(t))| \geq \epsilon.
\label{111}
\end{equation}
On the other hand,
\begin{equation}
  \begin{array}{ll}
	  \hspace{-.5in}|\sum_{j\not\in J^1(t)} a_{ij}(t)(x_j(t)-x_i(t))| & \vspace{.05in} \\ & \hspace{-1.6in} \leq \sum_{j \not\in J^1(t)} a_{ij}(t) |x_j(t)-x_i(t)|\vspace{.05in} \\
		& \hspace{-1.6in}\leq L \sum_{j \not\in J^1(t)} a_{ij}(t),
	\end{array}
\label{112}	
\end{equation}
where $L$ is defined in Eq. (\ref{DS-L}). Eqs. (\ref{111}) and (112) imply:
\begin{equation}
  \sum_{j \not\in J^1(t)} a_{ij}(t) \geq \epsilon / L.
\end{equation}
Thus, there exists $j \not\in J^1(t)$ such that
\begin{equation}
  a_{ij}(t) \geq \frac{\epsilon}{L(N-1)}.
\end{equation}
Now, from the definition of weak aperiodicity, we know that there exists $l \in \mathcal{V}$ such that $\min\{ a_{li}(t),a_{lj}(t) \} \geq \gamma a_{ij}(t) \geq \gamma \epsilon / L (N-1)$. Note that $i$ and $j$ are in different jets at time $t$. Thus, $l$ cannot be in the same jet with both $i$ and $j$ at time $t$. Therefore, at least one of $a_{li}(t),a_{lj}(t)$ indicates an interaction between a jet and its complement. Since both values are bounded below by $\gamma \epsilon / L (N-1)$, the sum of interactions between jets $J^k$'s and their complements is at least $\gamma \epsilon / L (N-1)$ at time $t$. On the other hand, from the D-S Theorem, we now that the total sum of flows between jets and their complements is finite over the infinite time interval. Since $\{ A(t) \}$ is of Class $\mathcal{P}^*$, the total sum of interactions between the jets and their complements must be finite as well. Hence, the number of times that the sum of interactions is at least $\gamma \epsilon / L (N-1)$, must be finite. Therefore, there are finite times in which an agent moves from a jet to another, and the jets become time-invariant after a finite time. It is straightforward to see that the time-invariant jets are connected components of the infinite flow graph.

%%%%%%%%%%%%%%%%%%%%%%%%%%%%%%%%%%%%%%%%%%%%%%%%%%%%%%%%%%%%%%%%%%%%
%%%%%%%%%%%%%%%%%%%%%%%%%%%%%%%%%%%%%%%%%%%%%%%%%%%%%%%%%%%%%%%%%%%%

\subsection{Self-Confident and Cut-Balanced Chains}

\begin{definition}\cite{Bolouki:12b}
  A chain $\{ A(t) \}$ of stochastic matrices is \textit{self-confident} with bound $\delta$ if $a_{ii}(t) \geq \delta$, $\forall i \in \mathcal{V}$, $\forall t \geq 0$.
\end{definition}
\begin{definition}\cite{Hend:11}
  A chain $\{ A(t) \}$ of stochastic matrices is \textit{cut-balanced} with bound $K$ if for every $\mathcal{V}_1 \subseteq \mathcal{V}$ and $t \geq 0$:
	\begin{equation}
	  \sum_{i \not\in \mathcal{V}_1}\sum_{j \in \mathcal{V}_1} a_{ij}(t) \leq K \sum_{i \in \mathcal{V}_1}\sum_{j \not\in \mathcal{V}_1} a_{ij}(t).
	\end{equation}
\end{definition}
\begin{proposition}\cite{Bolouki:12b,Touri:11c}
  If chain $\{ A(t) \}$ is self-confident and cut-balanced, then it is class-ergodic and the islands form the ergodic classes of $\{ A(t) \}$.
\end{proposition}
\begin{proof}
  Assume that $\{ A(t) \}$ has self-confidence and cut-balance properties with bounds $\delta$ and $K$ respectively. The chain being self-confident and cut-balanced, it in Class $\mathcal{P}^*$ (see \cite[Theorem 7]{Touri:11c} where self-confidence is referred to as strong aperiodicity). Thus, from Theorem \ref{IJFP-equivalent}, it is sufficient to show that for an arbitrary island $I$ and an arbitrary proper jet $J$ in $I$, we have $U(J,I \backslash J)=\infty$ (that is the infinite jet flow property holds island-wise). Indeed, if jet-limit $J^*$ exists, unboundedness of $U(J,I \backslash J)$ is immediately implied from unboundedness of $U(J^*,I \backslash J^*)$ in view of the definition of islands. Otherwise, there are infinitely many instants $t$ such that $J(t) \neq J(t+1)$. At every such $t$, there exists $i \in I$ such that $i \in (J(t) \backslash J(t+1)) \cup (J(t+1) \backslash J(t))$. Therefore, recalling (\ref{U_n}), $U_t(J,I\backslash J) \geq a_{ii}(t) \geq \delta$. Since there are infinitely many such times, $U(J,I \backslash J)$ is unbounded.
\end{proof}

%%%%%%%%%%%%%%%%%%%%%%%%%%%%%%%%%%%%%%%%%%%%%%%%%%%%%%%%%%%%%%%%%%%%
%%%%%%%%%%%%%%%%%%%%%%%%%%%%%%%%%%%%%%%%%%%%%%%%%%%%%%%%%%%%%%%%%%%%

\subsection{Balanced Asymmetric Chains}
\label{bolouki}
\begin{definition}\cite{Bolouki:12b}
  A chain $\{ A(t) \}$ of stochastic matrices is said be \textit{balanced asymmetric} with bound $M$, if for every subsets $\mathcal{V}_1,\mathcal{V}_2 \subseteq \mathcal{V}$ of the same cardinality, and for every $t \geq 0$:
	\begin{equation}
	  \sum_{i \not\in \mathcal{V}_1}\sum_{j \in \mathcal{V}_2} a_{ij}(t) \leq M \sum_{i \in \mathcal{V}_1}\sum_{j \not\in \mathcal{V}_2} a_{ij}(t).
	\end{equation}
\end{definition}
\begin{proposition}
Every balanced asymmetric chain is in Class $\mathcal{P}^*$.
\label{balanced-in-P*}
\end{proposition}

To prove Proposition \ref{balanced-in-P*}, we need the following lemma.

\begin{lemma}
Let $A$ be an $(N \times N)$ balanced asymmetric matrix with bound $M$. Then, there exists a permutation matrix $P_{N \times N}$ such that the product $PA$ is self-confident with bound $\delta = 4/(M N^2+4N-4)$.
\label{lem1}
\end{lemma}

\begin{proof}
Form a bipartite-graph $\mathcal{H}(\mathcal{V},\mathcal{E})$ from $A$ with $N$ nodes in each part. Let $\mathcal{V}_1$ and $\mathcal{V}_2$, each a copy of $\mathcal{V}$, be sets of nodes of the two parts of $\mathcal{H}$. For every $i \in \mathcal{V}_1$ and $j \in \mathcal{V}_2$, connect $i$ to $j$ if $a_{ij} \geq \delta = 4/(M N^2+4N-4)$. We wish to show that $\mathcal{H}$ has a perfect matching. By Hall's Marriage Theorem \cite[Theorem 5.2]{Bondy:76}, it suffices to show that for every subset $\mathcal{K} \subseteq \mathcal{V}_1$, we have $|D(\mathcal{K})| \geq |\mathcal{K}|$ where
\begin{equation}
  D(\mathcal{K}) = \{ j \in \mathcal{V}_2 | \exists i \in \mathcal{K} \text{ s.t. } (i,j) \in \mathcal{E} \}.
\end{equation}
Indeed, assume that on the contrary, there exists $\mathcal{K} \subseteq \mathcal{V}_1$ such that $k'=|D(\mathcal{K})| < |\mathcal{K}|=k$. Let $\mathcal{K} = \{ c_1,\ldots,c_k \}$ and $D(\mathcal{K}) = \{ d_1,\ldots,d_{k'} \}$. Define $\mathcal{K}' \subsetneq \mathcal{K}$ by $\mathcal{K}' = \{ c_1,\ldots,c_{k'} \}$. We now have:
\begin{equation}
\sum_{i \in \mathcal{K}'}\sum_{j \not\in D(\mathcal{K})} a_{ij} < k' (N-k') \delta \leq \delta N^2/4.
\label{lem1-1}
\end{equation}
On the other hand,
\begin{equation}
  \begin{array}{ll}
    \sum_{i \not\in \mathcal{K}'}\sum_{j \in D(\mathcal{K})} a_{ij} & \hspace{-.1in} \geq \sum_{i \in \mathcal{K} \backslash \mathcal{K}'}\sum_{j \in D(\mathcal{K})} a_{ij} \vspace{.05in} \\ & \hspace{-.1in} = (k-k')- \sum_{i \in \mathcal{K} \backslash \mathcal{K}'}\sum_{j \not\in D(\mathcal{K})} a_{ij}\vspace{.05in} \\ & \hspace{-.1in} \geq (k-k')-(k-k')(N-k')\delta \vspace{.05in} \\ & \hspace{-.1in} \geq 1 - (N-1)\delta.
  \end{array}
\label{lem1-2}
\end{equation}
Since $\mathcal{K}',D(\mathcal{K}) \subsetneq \mathcal{V}$ are of identical cardinalities, the balanced asymmetry property of $A$ together with (\ref{lem1-1}) and (\ref{lem1-2}) imply that
\begin{equation}
  1 - (N-1) \delta < \delta M N^2/4.
\end{equation}
Thus, $\delta > 4/(M N^2+4N-4)$, which is a contradiction. Therefore, $\mathcal{H}$ has a perfect matching and consequently, there exists a permutation $\tau$ such that $a_{\tau(i),i} \geq \delta$, $\forall i$. Thus, the permutation matrix $P$ with $e_{\tau(i)}$ as its $i$th row, where $e_j$ denotes a row vector of length $N$ with 1 in the $j$th position and 0 in every other position, is such that the product $PA$ is self-confident with $\delta$.
\end{proof}

\hspace{.1in} \textit{Proof of Proposition \ref{balanced-in-P*}:} Let $\{A(t)\}$ be a balanced asymmetric chain with bound $M$. Set: $\delta = 4/(M N^2+4N-4)$. We recursively define sequence $\{P(t)\}$ of permutation matrices as follows: From Lemma \ref{lem1}, we know that there exists a permutation matrix $P(0)$ such that the product $P(0)A(0)$ is self-confident with $\delta$. Find permutation matrix $P(t)$, $t \geq 1$, such that the product $P(t)A(t)P'(t-1)$ is self-confident with $\delta$. Note that the existence of $P(t)$ is implied by Lemma \ref{lem1}, taking into account the fact that the product $A(t)P'(t-1)$ is balanced asymmetric with bound $M$, since the columns of the product are a permutation of the columns of $A(t)$, itself a balanced asymmetric matrix with bound $M$. Hence, if we define chain $\{B(t)\}$ by:
\begin{equation}
  B(t) = P(t)A(t)P'(t-1),
\end{equation}
then, $\{B(t)\}$ has both the self-confidence and balanced asymmetry properties. Since balanced asymmetry is stronger than cut-balance, chain $\{B(t)\}$ is both self-confidence and cut-balanced. Thus, from \cite{Touri:11c}, we conclude that chain $\{B(t)\}$ belongs to the set $\mathcal{P}^*$. Furthermore, it is straightforward to show that if $\{ \pi(t) \}$ is an absolute probability sequence adapted to chain $\{B(t)\}$, then $\{ \pi(t) P(t-1) \}$, where $P(-1) = I_{N \times N}$, is an absolute probability sequence adapted to chain $\{A(t)\}$. This immediately implies that $\{A(t)\} \in \mathcal{P}^*$.
\qed

The class property $\mathcal{P}^*$ implies that absolute probabilities are uniformly bounded away from zero, and as a result, that there is no $J^0$ in the jet decomposition of the D-S Theorem. Therefore, we again consider only $J^1,\ldots,J^c$ as the jet decomposition.
\begin{proposition}
  If $\{ A(t) \}$ is balanced asymmetric, then the cardinality of each jet in the jet decomposition of the D-S Theorem, becomes time-invariant after a finite time.
\label{time-invariant-jet-cardinality}
\end{proposition}
\begin{proof}
  Let $\{ A(t) \}$ be balanced asymmetric with bound $M$. It suffices to show that there are finite times in which cardinality of a jet, in the jet decomposition of the D-S Theorem, increases by at least 1. In the following, we see what happens when the cardinality of a jet, say $J^k$, increases. Assume that for a fixed $t \geq 0$, we have $|J^k(t+1)| > |J^k(t)|$. For an arbitrary $i \in J^k(t+1)$, let $\mathcal{T} \subsetneq J^k(t+1)$ be such that $i \not\in \mathcal{T}$ and $|\mathcal{T}| = |J^k(t)|$. Thus by the balanced asymmetry property,
	\begin{equation}
	  \begin{array}{ll}
	  \sum_{j \in J^k(t)} a_{ij}(t) & \hspace{-.1in} \leq \sum_{l\not\in \mathcal{T}}\sum_{j \in J^k(t)} a_{lj}(t) \vspace{.05in} \\ & \hspace{-.1in} \leq M \sum_{l\in \mathcal{T}}\sum_{j \not\in J^k(t)} a_{lj}(t) \vspace{.05in} \\ & \hspace{-.1in} \leq  M \sum_{l \in J^k(t+1)}\sum_{j \not\in J^k(t)} a_{lj}(t).
		\end{array}
	\end{equation}
	Therefore,
	\begin{equation}
	  \begin{array}{l}
	    \hspace{-.2in}\sum_{i\in J^k(t+1)}\sum_{j \in J^k(t)}a_{ij}(t)\vspace{.05in} \\ \hspace{.1in} \leq |J^k(t+1)|.M \sum_{i \in J^k(t+1)}\sum_{j \not\in J^k(t)} a_{ij}(t).
		\end{array}
	\label{1001}
	\end{equation}
	On the other hand,
	\begin{equation}
	  \begin{array}{l}
	    \hspace{-.2in}\sum_{i\in J^k(t+1)}\sum_{j \in J^k(t)}a_{ij}(t)\vspace{.05in} \\ \hspace{.1in}= |J^k(t+1)| - \sum_{i \in J^k(t+1)}\sum_{j \not\in J^k(t)} a_{ij}(t).
		\end{array}
	\label{1002}
	\end{equation}
	Eqs. (\ref{1001}) and (\ref{1002}) together imply:
	\begin{equation}
	  \hspace{0in}\sum_{i \in J^k(t+1)}\sum_{j \not\in J^k(t)} a_{ij}(t) \geq \frac{|J^k(t+1)|}{1+M |J^k(t+1)|} \geq \frac{1}{1+M}.
	\label{1003}
	\end{equation}
	Once again, since the cumulative interaction between $J^k$ and $\bar{J}^k$ must be finite over the infinite time interval because of the D-S Theorem and in view of the class property $\mathcal{P}^*$, (\ref{1003}) can occur only for finitely many times $t$, and this completes the proof.
\end{proof}
An immediate corollary of Proposition \ref{time-invariant-jet-cardinality} is as follows.
\begin{corollary}\cite{Bolouki:12b}
  Consider a multi-agent system with dynamics (\ref{DS-modelss}), where $\{ A(t) \}$ is balanced asymmetric. Then, $z_i(t)$ converges for every $i \in \mathcal{V}$, as $t$ goes to infinity, where $z_i(t)$ is the $i$th least value among $x_1(t),\ldots,x_N(t)$.
\end{corollary}
\begin{definition}\cite{Touri:3}
  A chain $\{ A(t) \}$ of stochastic matrices is said to have the \textit{absolute infinite flow property}, if for every jet $J$ in $\mathcal{V}$ with a time-invariant size, $U(J,\mathcal{V} \backslash J)$ is unbounded.
\end{definition}
\begin{proposition}\cite{Bolouki:12b}
  If $\{A(t)\}$ is balanced asymmetric, then, $\{A(t)\}$ is class-ergodic if and only if the absolute infinity property holds over each island of $\{A(t)\}$. Furthermore, in case of class-ergodicity, islands are the ergodic classes of $\{A(t)\}$.
\end{proposition}
\begin{proof}
  From Proposition \ref{balanced-in-P*}, we know that $\{A(t)\} \in \mathcal{P}^*$. Therefore, taking advantage of Theorem \ref{IJFP-equivalent}, it suffices to show that absolute infinite flow and infinite jet-flow properties are equivalent on each island. Obviously, the former is implied by the latter. We prove the converse as follows: Let the absolute infinite flow property hold over each island. Assume that $I$ is an arbitrary island of $\{A(t)\}$ and $J$ is an arbitrary jet in $I$. If the cardinality of jet $J$ becomes time-invariant after a finite time, unboundedness of $U(J,I \backslash J)$ is immediately implied from the absolute infinite flow property over $I$. Otherwise, the cardinality of $J$ increases infinitely many times by at least 1. In this case, from the proof of Proposition \ref{time-invariant-jet-cardinality} (see (\ref{1003})), we know that $V(J,\mathcal{V} \backslash J)$ is unbounded, and consequently $U(J,\mathcal{V} \backslash J)$ is unbounded following Lemma \ref{infinite V-U}. Moreover,
	\begin{equation}
	  U(J,\mathcal{V} \backslash J) + U(I \backslash J, \mathcal{V} \backslash I) = U(J, I \backslash J) + U(I, \mathcal{V}\backslash I),
	\end{equation}
	and since $U(I, \mathcal{V}\backslash I)$ is bounded because $I$ is an island, unboundedness of $U(J,\mathcal{V} \backslash J)$ implies that $U(J, I \backslash J)=\infty$. This completes the proof.
\end{proof}
\begin{corollary}\cite{Bolouki:12b}
  If chain $\{ A(t) \}$ is balanced asymmetric, then it is ergodic if and only if it has the absolute infinite flow property.
\end{corollary}
%%%%%%%%%%%%%%%%%%%%%%%%%%%%%%%%%%%%%%%%%%%%%%%%%%%%%%%%%%%%%%%%%%%%
%%%%%%%%%%%%%%%%%%%%%%%%%%%%%%%%%%%%%%%%%%%%%%%%%%%%%%%%%%%%%%%%%%%%
%%%%%%%%%%%%%%%%%%%%%%%%%%%%%%%%%%%%%%%%%%%%%%%%%%%%%%%%%%%%%%%%%%%%

\section{A Geometric Approach towards Consensus Algorithms}
\label{geo-approach-discrete}
In this section, we introduce a geometric framework for a general linear consensus algorithm, that not only interprets the notions of jets and the ocean as explained in the previous sections, but serves an alternative proof of our results stated in the previous sections, and furthermore, as will be shown in the next section, extends them naturally to the continuous time case.

Let $\Phi(t,\tau)$, $t,\tau \geq 0$, be the state transition matrix of discrete time model (\ref{DS-modelss}), i.e.,
\begin{equation}
  \Phi(t,\tau) = A(t-1)A(t-2) \cdots  A(\tau).
  \label{phiii}
\end{equation}
Therefore,
\begin{equation}
  x(t) = \Phi(t,\tau) x(\tau), \, \forall t,\tau \geq 0.
\end{equation}
For every $t \geq \tau \geq 0$, assume that $C_{t,\tau}$ is the convex hull of the columns of $\Phi'(t,\tau)$. Note that each column of $\Phi'(t,\tau)$ is a stochastic vector representing a point in $\mathbb{R}^N$, and $C_{t,\tau}$ is a polytope in $\mathbb{R}^N$ if we consider points and segments in $\mathbb{R}^N$ as polytopes with one and two vertices respectively.
\begin{lemma}
  For every $t_2 \geq t_1 \geq \tau$, we have: $C_{t_2,\tau} \subset C_{t_1,\tau}$, i.e., polytope $C_{t_1,\tau}$ contains polytope $C_{t_2,\tau}$.
	\label{DS-convex-hull}
\end{lemma}
\begin{proof}
  From (\ref{phiii}), we have:
	\begin{equation}
	  \Phi'(t_2,\tau) = \Phi'(t_1,\tau)\Phi'(t_2,t_1).
	\end{equation}
	Since $\Phi(t_2,t_1)$ is a stochastic matrix, each column of $\Phi'(t_2,\tau)$ is a convex combination of columns of $\Phi'(t_1,\tau)$. Therefore, every column of $\Phi'(t_2,\tau)$ lies in $C_{t_1,\tau}$, and the lemma is proved.
\end{proof}

Lemma \ref{DS-convex-hull} shows that for a fixed $\tau \geq 0$, $C_{t,\tau}$ shrinks as $t$ grows. A projection of nested polytopes $C_{t,\tau}$'s on a two-dimensional space is shown in Fig \ref{Fig2-1}.

\begin{figure}[h]
  \begin{center}
  \captionsetup{justification=centering}
    \includegraphics[scale=1]{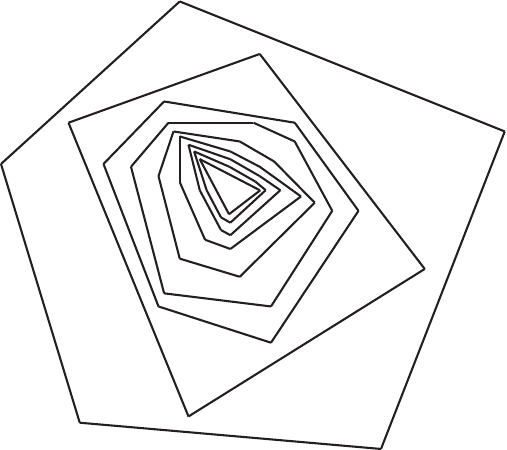}
    \caption{An example of nested polygons converging to a triangle.}
    \label{Fig2-1}
  \end{center}
\end{figure}

It is to be noted that when underlying chain $\{A(t)\}$ of dynamics (\ref{DS-modelss}) is ergodic, the nested polygons will converge to a point in $\mathbb{R}^N$. In general, one concludes that for every $\tau \geq 0$, $\lim_{t \rightarrow \infty} C_{t,\tau}$ exists and is also a polytope in $\mathbb{R}^N$. Let $C_{\tau}$ be the limiting polytope with $c_{\tau}$ vertices. It is clear that $1 \leq c_{\tau} \leq N$. One can show that $c_{\tau}$ is a non-decreasing function of $\tau$ (see \cite[Part VIII-B]{Bolouki:14b}) and becomes constant after some finite time. We assume, without loss of generality, that $c_{\tau}$ is equal to constant $c$, $\forall \tau \geq 0$. It is worth mentioning that the choice of letter $c$ here, that also represents the number of jets in the jet decomposition of the Sonin's D-S Theorem in this paper, for the number of vertices of limiting polytope $C_0$, is not accidental, as it will be shown, in the current section, that the two numbers are equal.

Let $v_1,\ldots,v_c$ be the $c$ vertices of $C_0$. Assume that $\{0_t\}$ is a sequence of agents, i.e., $0_t \in \mathcal{V}$ for every $t \geq t_0$.
\begin{theorem}
  If sequence $\{0_t\}_{t \geq 0}$, $0_t \in \mathcal{V}$, $\forall t \geq 0$, is such that the distance between $\Phi'_{0_t}(t,0)$ and set $\{v_1,\ldots,v_c\}$ does not converge to zero as $t$ grows large, then:
	\begin{equation}
	  \inf \{ \pi_{0_t}(t) \, | \, t \geq 0 \} = 0.
	\end{equation}
\label{ocean detected}
\end{theorem}
\begin{proof}
	We know that vector $v_i$, $1 \leq i \leq c$, lies outside of the convex hull of vectors $v_j$'s, $j \neq i$. Let $w_i$ be the nearest point to $v_i$, on the convex hull of $v_j$'s, $j \neq i$. For a small $\epsilon' > 0$, draw a hyperplane, distant $\epsilon'$ from $v_i$, crossing segment $v_iw_i$ and orthogonal to it. Let $u(t) \triangleq  \Phi'_{0_t}(t,0)$. For a sufficiently small $\epsilon'$, there exists a subsequence of $\{u(t)\}$ such that $v_i$ and the elements of the subsequence lie on opposite sides of the hyperplane for every $i$, $1 \leq i \leq c$. Otherwise, the distance between $\{u(t)\}$ and set $\{v_1,\ldots,v_c\}$ would converge to zero.
	%Let $\{u(t_i)\}_{i=1,2,\ldots}$ be such subsequence.
	Define:
	\begin{equation}
	  \epsilon'_1 \triangleq \min \{ |v_i-w_i| \, | \, 1 \leq i \leq c \},
	\end{equation}
	and:
	\begin{equation}
	  \epsilon \triangleq \min\{ \epsilon' , \epsilon'_1/4 \},
	\end{equation}
	and for an arbitrary constant $\delta$, $0 < \delta < 1$, let:
	\begin{equation}
	  \epsilon_1 \triangleq \delta\epsilon /(2N)
	\end{equation}
	We summarize the rest of the proof, since it is very similar to the proof of \cite[Lemma 7]{Bolouki:14b}, from (25) to (35). We know that for a sufficiently large time $T \geq 0$, if $t \geq T$, every vector in $C_{t,0}$ lies within an $\epsilon_1$-distance of $C_{0}$. For every $i$, $1 \leq i \leq c$, draw a hyperplane $l_i$, parallel to the hyperplane drawn previously, distant $\epsilon$ from $v_i$, crossing segment $v_iw_i$. Draw also a hyperplane $m_i$, parallel to $l_i$, on the other side of $v_i$, distant $\epsilon_1$ from $v_i$ (see Fig. \ref{Figocean}).
	
		\begin{figure}[h]
  		\begin{center}
  				\captionsetup{justification=centering}
    				\vspace{-.7in}
    				\includegraphics[scale=.4]{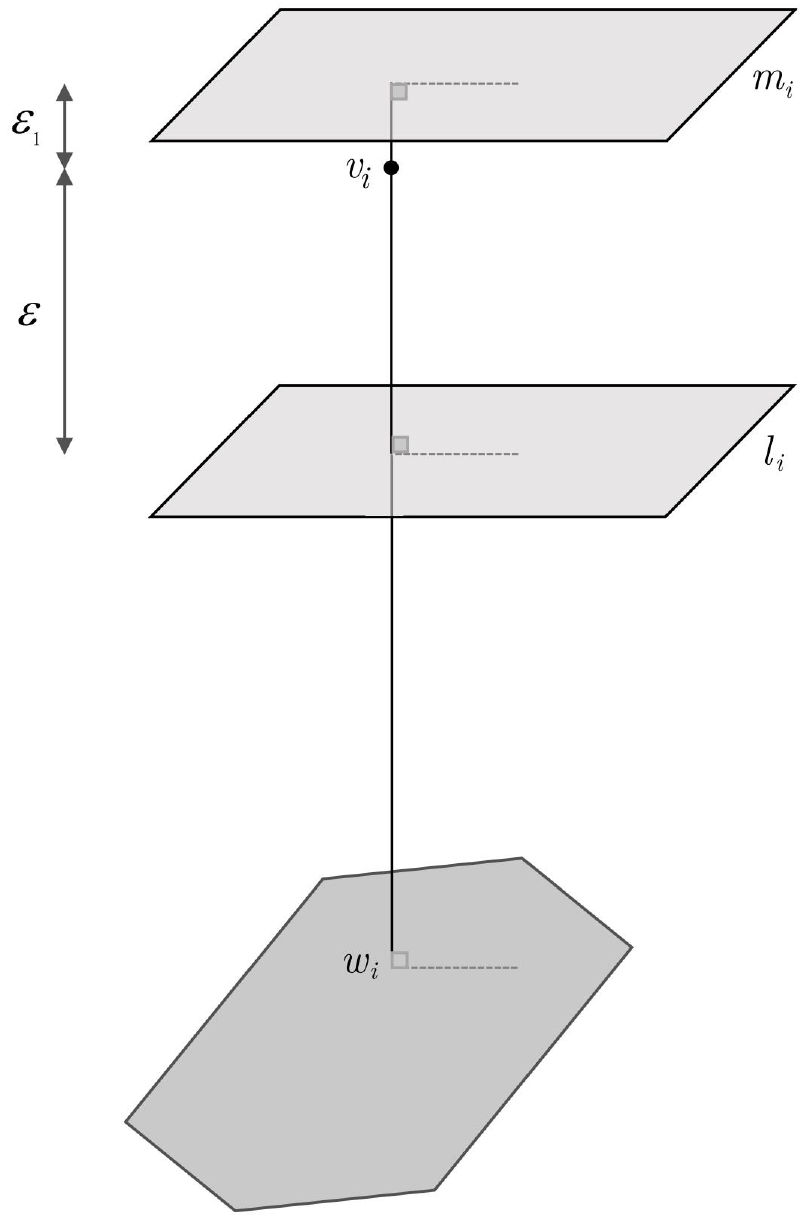}
    				\vspace{-.4in}
    				\caption{Planes $l_i$ and $m_i$ are orthogonal to segment $v_iw_i$.}
		\label{Figocean}
			\end{center}
		\end{figure}
	
	Define for every $i$, $1 \leq i \leq c$:
	\begin{equation}
	  S^i = \{ j \in \mathcal{V} \, | \, \Phi'_j(T,0) \text{ lies in strip margined by } l_i,m_i \}.
	\end{equation}
	One can show that, $S^i$'s, $1 \leq i \leq c$, are disjoint non-empty sets. Define also:
	\begin{equation}
	  S^0 = \mathcal{V}\, \backslash \bigcup_{j=1}^c S^j.
	\end{equation}
	As mentioned above, there exists a subsequence of $\{u(t)\}$ such that $v_i$ and the elements of the subsequence lie on the opposite sides of $l_i$ (note that $\epsilon \leq \epsilon'$) for every $i = 1,\ldots,c$. Without loss of generality, assume that $\{u(T)\}$ belongs to that subsequence (otherwise, choose $T_1>T$ such that $u(T_1)$ belongs to that subsequence, and replace $T$ by $T_1$ in the argument). Hence, $S^0 \neq \emptyset$, and $S^i$ 's partition agent set $\mathcal{V}$. Similar to the proof of \cite[Lemma 7]{Bolouki:14b}, we have the following inequality (equivalence of inequality (35) in \cite{Bolouki:14b}):
%%%%%%%%%%%%%%%
\begin{comment}	
  	Assume that for $i=1,\ldots,d$, $\{i_t\}_{t \geq T}$ is a sequence of agents such that $|\Phi'_{i_t}(t,t_0)-v_i|<\epsilon$. We have:
       \begin{equation}
        \Phi'_{i_t}(t,0) = \Phi'(T,0) \Phi'_{i_t}(t,T) = \sum_{j \in \mathcal{V}}\Phi_{i_t ,j}(t,T) \Phi_j(T,0).
       \end{equation}
	Therefore, for every, $i$, $1 \leq i \leq d$:
	\begin{equation}
	  \sum_{j\not\in S^i} \Phi_{i_t, j}(t,T) \leq \delta/N,
	\end{equation}
	and
	\begin{equation}
	  \sum_{j \in S^i} \Phi_{i_t, j}(t,T) \geq 1-\delta/N.
		\label{tend}
	\end{equation}
	We know that $\Phi'_{i_t}(t,T)$ tend to lie in or on polytope $C_T$ as $t$ grows large. Thus, from (\ref{tend}), for at least one of the vertices of $C_T$, say $u_i$, we must have:
	\begin{equation}
	  \sum_{j \in S^i} (u_i)_j \geq 1-\delta/N.
		\label{w}
	\end{equation}
	Note that since $\sum_{j \in \mathcal{V}} (u_i)_j =1$, $u_i$'s are distinct vertices of $C_T$, as none can satisfy two different such inequalities as (\ref{w}). Therefore, $C_T$ has at least $d$ distinct vertices. Since $d$ was assumed to be the maximum number of vertices that a polytope $C_{\tau}$, $\tau \geq 0$ could have, $C_T$ has exactly $d$ vertices: $u_1,\ldots,u_d$. For every $i$, $1 \leq i \leq d$,
\end{comment}
%%%%%%%%%%%%%%%
	\begin{equation}
	  \sum_{j \not\in S^i} (u_i)_j \leq 2\delta/(2N+1)< \delta/N.
	\end{equation}
	Consequently,
	\begin{equation}
	  \sum_{j \in S^0} (u_i)_j \leq \sum_{j \not\in S^i} (u_i)_j < \delta/N.
	\end{equation}
	Thus, for every $i \in \mathcal{V}$ and $j \in S^0$:
	\begin{equation}
	  \inf \{\Phi_{i,j}(t,T) \, | \, t \geq T\} < \delta / N.
	\end{equation}
		Consequently,
	\begin{equation}
	  \inf \{\sum_{i \in \mathcal{V},j\in S^0}\Phi_{i,j}(t,T) \, | \, t \geq T\} < N\delta/N = \delta.
	\label{143}
	\end{equation}
       Since we have:	
	\begin{equation}
	  \pi_j(T) = \pi(t) \Phi^j(t,T) = \sum_{i \in \mathcal{V}} \pi_i(t) \Phi_{i,j}(t,T) \leq \sum_{i \in \mathcal{V}}\Phi_{i,j}(t,T),
	\end{equation}
       from (\ref{143}) we conclude that:
	\begin{equation}
	  \pi_j(T) < \delta, \, \forall j \in S^0.
	\end{equation}
	We recall that $S^0$ includes one of the elements of sequence $\{0_t\}$, i.e., $0_T$. Hence, $\pi_{0_T}(T) < \delta$. Recall also that $\delta$ was chosen arbitrarily. By letting $\delta$ go to zero, we conclude that $\inf \{ \pi_{0_t}(t) \, | \, t \geq 0 \} = 0$, and the theorem is proved.
\end{proof}

\begin{remark}
\label{remm}
  We explain, in the following, that there is a one-to-one correspondence between the vertices of limiting polytope $C_0$ and jets $J^1,\ldots,J^c$ of the Sonin's jet decomposition.
  
  Recall, from Section \ref{D-S Theorem}, that how we employed the absolute probability sequence of chain $\{A(t)\}$ to construct a forward propagating Markov chain from the given backward one. Now, let $J^k$ be an arbitrary jet among $J^1,\ldots,J^c$. Let, also, $\{k_t\}$ be a sequence inside jet $J^k$, i.e., $k_t \in J^k(t)$, $\forall t \geq 0$. Since, due to the D-S Theorem, $\lim_{t \rightarrow \infty} x_{k_t}(t)$ exists irrespective of what $x(0)$ is, $\lim_{t \rightarrow \infty} \Phi'_{k_t}(t,0)$ exists as well, and is irrespective of how the sequence is chosen. We aim to show that $\lim_{t \rightarrow \infty} \Phi'_{k_t}(t,0)$ is one of $v_1,\ldots,v_c$. Since the volume of $J^k(t)$ converges to a non-zero constant, as $t \rightarrow \infty$, one can form a sequence $\{k_t\}$ inside jet $J^k$, i.e., $k_t \in J^k(t)$, $\forall t \geq 0$, such that:
  \begin{equation}
    \liminf \{ \pi_{k_t}(t) \, | \, t \geq 0 \} > 0.
    \label{djk}
  \end{equation}
  One way to form such a sequence is to pick, at each time instant, the cup in $J^k$ that has the maximum volume. From Theorem \ref{ocean detected} and inequality (\ref{djk}), we conclude that the distance between $\Phi'_{k_t}(t,0)$ and set $\{v_1,\ldots,v_c\}$, the vertices of limiting polytope $C_0$, must vanish as $t$ grows large. Thus, $\lim_{t \rightarrow \infty} \Phi'_{k_t}(t,0)$ is belongs to set $\{v_1,\ldots,v_c\}$.
  
  It is also clear that if sequences $\{s_t\}$ and $\{k_t\}$ are in two disjoint jets $J^s$ and $J^k$ respectively, $\lim_{t \rightarrow \infty} \Phi'_{s_t}(t,0)$ and $\lim_{t \rightarrow \infty} \Phi'_{k_t}(t,0)$ cannot converge to the same vertex of $C_0$, since otherwise, merging the two jets would violate the uniqueness of the Sonin's jet decomposition.
\end{remark}

%%%%%%%%%%%%%%%%%%%%%%%%%%%%%%%%%%%%%%%%%%%%%%%%%%%%%%%%%%%%%%%%%%%%
%%%%%%%%%%%%%%%%%%%%%%%%%%%%%%%%%%%%%%%%%%%%%%%%%%%%%%%%%%%%%%%%%%%%
%%%%%%%%%%%%%%%%%%%%%%%%%%%%%%%%%%%%%%%%%%%%%%%%%%%%%%%%%%%%%%%%%%%%

\section{Consensus in the Continuous Time Case}
\label{cons-cont}

One may define a general linear consensus algorithm in continuous time as follows:
\begin{equation}
  \dot{x} = A(t)x(t), \, t \geq 0,
  \label{DS-systemss}
\end{equation}
where $x(t)$ is the vector of opinions at each time instant $t \geq 0$ and $\{A(t)\}$ is the underlying chain of the system. It is assumed that each matrix of underlying chain $A(t)$ has zero row sum and non-negative off-diagonal elements and each element $a_{ij}(t)$ of $A(t)$ is a measurable function. These constraints on the underlying chain suggest a view of A(t) as the evolution of the intensity matrix of a time inhomogeneous Markov chain. We shall use in this section, a continuous time version of the geometric framework developed in Section \ref{geo-approach-discrete}, in convergence analysis of agents in a network with continuous time dynamics (\ref{DS-systemss}), particularly when underlying chain $\{A(t)\}$ is in a continuous time version of Class $\mathcal{P}^*$.

Let $\Phi(t,\tau)$, $t,\tau \geq 0$, represent the state transition matrix of system associated with (\ref{DS-systemss}), i.e., 
\begin{equation}
  x(t) = \Phi(t,\tau) x(\tau), \, \forall t \geq \tau \geq 0.
  \label{mgeneral-1}
\end{equation}
Note that similar to the discrete time case, $\Phi(t,\tau)$ is a stochastic matrix for every $t \geq \tau \geq 0$. More specifically, $\Phi_{i,j}(t,\tau)$ can be considered as transition probability of a backward propagating inhomogeneous Markov chain. In particular, for every $t_2 \geq t_1 \geq \tau \geq 0$, we have:
  	\begin{equation}
    		\Phi_{i,j}(t_2,\tau) = \sum_{k} \Phi_{i,k}(t_2,t_1) \Phi_{k,j}(t_1,\tau),
  	\end{equation}
	with the conditions:
  	\begin{equation}
    		\Phi_{i,j}(t,\tau) \geq 0,
  	\end{equation}
  	\begin{equation}
    		\sum_{j} \Phi_{i,j}(t,\tau) = 1,
  	\end{equation}
  	\begin{equation}
    		\Phi_{i,j}(t,t) = \delta_{ij},
  	\end{equation}
  	where $\delta_{ij}$ is the Kronecker symbol.
	
	Underlying chain $\{A(t)\}$ is said to be ergodic if for every $\tau \geq 0$, $\Phi(t,\tau)$ converges to a matrix with equal rows as $t \rightarrow \infty$. Similar to the discrete time case, ergodicity of $\{A(t)\}$ is equivalent to the occurrence of unconditional consensus in (\ref{DS-systemss}). Moreover, $\{A(t)\}$ is class-ergodic if for every $\tau \geq 0$, $\lim_{t \rightarrow \infty}\Phi(t,\tau)$ exists, but with possibly distinct rows. Chain $\{ A(t) \}$ is class-ergodic if and only if multiple consensus occurs in (\ref{DS-systemss}) unconditionally. Recall that the associated state transition matrix associated with (\ref{DS-systemss}) can be expressed via the Peano-Baker series (see \cite[Section 1.3]{Brockett:70}):
\begin{equation}
  \begin{array}{ll}
		  \Phi(t,\tau) = I_{N\times N} & \hspace{-.1in} + \int_{\tau}^t A(\sigma_1)d\sigma_1 \vspace{.05in}\\
			                     & \hspace{-.1in}+ \int_{\tau}^t A(\sigma_1) \int_{\tau}^{\sigma_1} A(\sigma_2) d\sigma_2 d\sigma_1 \vspace{.05in}\\
					    & \hspace{-.1in}+ \int_{\tau}^t A(\sigma_1) \int_{\tau}^{\sigma_1} A(\sigma_2) \int_{\tau}^{\sigma_2} A(\sigma_3) d\sigma_3 d\sigma_2 d\sigma_1 \vspace{.05in}\\
					    & \hspace{-.1in}+ \cdots,
  \end{array}
  \label{state transition matrix-1}
\end{equation}
where $I_{N \times N}$ denotes the $N\times N$ identity matrix. Remember that state transition matrix $\Phi(t,\tau)$ is invertible for every $t \geq \tau \geq 0$.

Furthermore, once again, based on \cite{Kolmo:36}, we know that for every state transition matrix $\Phi(t,\tau)$, $t,\tau \geq 0$, there exists an absolute probability sequence $\{ \pi(t) \}$, $t \geq 0$, such that:
\begin{equation}
  \pi(\tau) = \pi(t) \Phi(t,\tau), \, \forall t,\tau \geq 0.
\end{equation}
Having recalled the state transition matrix and the absolute probability sequence for the continuous time model (\ref{DS-systemss}), we can now carry out a continuous time version of the geometric framework developed in Section \ref{geo-approach-discrete}. Once again, for every $t \geq \tau \geq 0$, assume that $C_{t,\tau}$ is the convex hull of columns of $\Phi'(t,\tau)$, or equivalently transposed rows of $\Phi(t,\tau)$. Remember that each column of $\Phi'(t,\tau)$ is a stochastic vector as in the discrete time case. Now, note that Lemma \ref{DS-convex-hull} holds for the continuous time as well, since its proof remains valid assuming that the time indices refer to continuous time. Therefore, we again assume that limiting polytopes $C_{\tau}$'s, $\tau \geq 0$, exist. Let $c_{\tau}$ be the number of vertices of $C_{\tau}$. We show in the following that, $c_{\tau}$, $\tau \geq 0$, is constant (unlike the discrete time case in which $c_{\tau}$ was monotonic increasing with respect to $\tau$). Assume that $\tau_2 \geq \tau_1 \geq 0$ are two arbitrary time instants. We wish to show that $c_{\tau_1}=c_{\tau_2}$. Define linear operator $\phi_{\tau_2,\tau_1}:\mathbb{R}^N \rightarrow \mathbb{R}^N$ by:
  		\begin{equation}
    			\phi_{\tau_2,\tau_1}(v) \triangleq \Phi'(\tau_2,\tau_1)v, \, \forall v \in \mathbb{R}^N.
    		\label{operator-1}
  		\end{equation}
		Note now that from properties of state transition matrices, for $t \geq \tau_2 \geq \tau_1 \geq 0$, we have:
  		\begin{equation}
    			\Phi'(t,\tau_1) =  \Phi'(\tau_2,\tau_1) \Phi'(t,\tau_2).
  		\label{nested-12}
  		\end{equation}
		Therefore, in view of (\ref{nested-12}) by taking $t$ to infinity, the vertices of $C_{\tau_2}$ are uniquely mapped to vectors in $\mathbb{R}^N$ which because of the linearity of map (\ref{operator-1}), will play the role of vertices for the generation of convex hull $C_{\tau_1}$. Also, it is not difficult to show that the images of vertices of $C_{\tau_2}$ must remain vertices of $C_{\tau_1}$, for if one of the images of a vertex of $C_{\tau_2}$, say $v$,  turned out to be a convex combination of other vertices of $C_{\tau_1}$, this would also be true for the inverse images of these vertices (also vertices of $C_{\tau_2}$ due to invertibility of matrix $\Phi'(\tau_2,\tau_1)$), and $v$ would then fail to be a vertex of $C_{\tau_2}$. In conclusion, $C_{\tau_1}$ and $C_{\tau_2}$ will have the same number of vertices, and (\ref{operator-1}) constitutes a one to one map between corresponding pairs of vertices. One may now use the same argument to extend Theorem \ref{ocean detected} to the continuous time case while $t_0$, the initial time in Theorem \ref{ocean detected}, can be chosen arbitrarily here (recall that for Theorem \ref{ocean detected} to be true, $C_{t_0}$ must have had the maximum number of vertices among all $C_{\tau}$'s, and since $c_{\tau}$ is constant for $\tau \geq 0$ in the continuous time case, $t_0$ can be chosen arbitrarily).

We now aim to take advantage of Theorem \ref{ocean detected} to address the limiting behavior of system (\ref{DS-systemss}) when underlying chain $\{ A(t) \}$ is in Class $\mathcal{P}^*$.
\begin{lemma}
  For every $j \in \mathcal{V}$,
	\begin{equation}
	  \pi_j(\tau) \leq \inf \left\{\sum_{i \in \mathcal{V}}\Phi_{i,j}(t,\tau) \, | \, t \geq \tau\right\}.
	\end{equation}
	\label{DS-inf}
\end{lemma}
\begin{proof}
  Obvious, since for every $t \geq \tau$:
	\begin{equation}
	  \pi_j(\tau) = \pi(t) \Phi^j(t,\tau) = \sum_{i \in \mathcal{V}} \pi_i(t) \Phi_{i,j}(t,\tau) \leq \sum_{i \in \mathcal{V}}\Phi_{i,j}(t,\tau).
	\end{equation}
\end{proof}

Adopting the same definition of Class $\mathcal{P}^*$ as in the discrete time case (see Section \ref{Touri}), we state the following lemma.

\begin{lemma}
  A state transition matrix $\Phi(t,\tau)$, $t,\tau \geq 0$, associated with (\ref{DS-systemss}), is in Class $\mathcal{P}^*$ if and only if for every $j \in \mathcal{V}$:
	\begin{equation}
	  \inf \left\{\sum_{i \in \mathcal{V}}\Phi_{i,j}(t,\tau) \, | \, t \geq \tau\right\} > 0.
	\end{equation}
	\label{DS-pstar}
\end{lemma}
\begin{proof}
  The \textit{only if} part is an immediate result of Lemma \ref{DS-inf}, and the \textit{if} part is a result of the way the existence of the absolute probability sequence can be obtained in \cite{Kolmo:36} by always choosing to initialize agent probabilities on finite intervals with a uniform distribution.
\end{proof}

Let the infinite flow graph of a continuous time chain $\{A(t)\}$ is defined according to Definition \ref{infinite flow graph-1} by replacing summation with integral. The following theorem describes the convergence properties of system (\ref{DS-systemss}) when the underlying chain is in Class $\mathcal{P}^*$.

\begin{theorem}
\label{DS-cont}
  If state transition matrix $\Phi(t,\tau)$, $t,\tau \geq 0$, is in Class $\mathcal{P}^*$, then multiple consensus occurs unconditionally in system (\ref{DS-systemss}). Moreover, the number of consensus clusters is equal to the number of the components of the infinite flow graph of the transition chain. In particular, consensus occurs unconditionally if and only if the infinite flow property holds.
\end{theorem}

The following theorem clarifies that Theorem \ref{DS-cont} generalizes continuous time consensus results of \cite{Hend:11}.

\begin{theorem}
  If transition chain $\{A(t)\}$ in (\ref{DS-systemss}) is cut-balanced, then state transition matrix $\Phi(t,\tau)$, $t \geq \tau \geq 0$, is in Class $\mathcal{P}^*$.
\end{theorem}
\begin{proof}
  Let $\{A(t)\}$ be cut-balanced with bound $K$. Assume that $\Phi(t,\tau)$, $t \geq \tau \geq 0$, is the state transition matrix associated with (\ref{DS-systemss}). In view of Lemma \ref{DS-pstar}, our aim is to show that: $1/N e' \Phi_A(t,\tau) \geq p^*e'$, for some $p^* > 0$, where $e' = \begin{bmatrix} 1 & \cdots & 1 \end{bmatrix}$, and the inequality is to be understood element-wise.
	
	Assume that $\alpha = \sup \{-a_{ii}(t') \, | \, i \in \mathcal{V}, \tau \leq t' \leq t\}$. Let chain $B$ be such that $B(t') = A(t') + 2\alpha I$, $\forall \tau \leq t' \leq t$, where $I$ is the identity matrix. It is easy to verify that:
	\begin{equation}
	  \Phi_B(t,\tau) = e^{2\alpha (t-\tau)} \Phi_A(t,\tau).
	\end{equation}
	Moreover, by construction, on-diagonal elements of $B(t')$, $\tau \leq t' \leq t$, are greater than or equal to $\alpha$. Note that $B(t')$ ($\tau \leq t' \leq t$) is not a stochastic matrix; instead each of its rows sums up to $2 \alpha$. We calculate in the following, $1/N e' \Phi_B(t,\tau)$. Therefore, from the Peano-Baker series (\ref{state transition matrix-1}), the expression:
	\begin{equation}
	  \frac{1}{N}e' \int_{\tau}^t B(\sigma_1) \int_{\tau}^{\sigma_1} B(\sigma_2)\cdots \int_{\tau}^{\sigma_{k-1}}B(\sigma_k)d\sigma_k \cdots d\sigma_1
		\label{expression}
	\end{equation}
	is of interest. Expression (\ref{expression}) is equal to:
	\begin{equation}
	  \frac{(2\alpha)^k}{N}e' \int_{\tau}^t \frac{B(\sigma_1)}{2\alpha} \int_{\tau}^{\sigma_1} \frac{B(\sigma_2)}{2\alpha}\cdots \int_{\tau}^{\sigma_{k-1}}\frac{B(\sigma_k)}{2\alpha}d\sigma_k \cdots d\sigma_1,
		\label{expression1}
	\end{equation}
	which is also equal to:
	\begin{equation}
	  (2\alpha)^k \int_{\tau}^t\int_{\tau}^{\sigma_1}\int_{\tau}^{\sigma_{k-1}} \frac{1}{N}e' \frac{B(\sigma_1)}{2\alpha} \frac{B(\sigma_2)}{2\alpha}\cdots\frac{B(\sigma_k)}{2\alpha}d\sigma_k \cdots d\sigma_1.
		\label{expression2}
	\end{equation}
	Note that $B(t')/2\alpha$ is a sequence of transition matrices which generates a Markov chain which is both cut-balanced and self-confident, and hence in Class $\mathcal{P}^*$ (\cite[Theorem 7]{Touri:11c}).  As a result, there exists a positive $p^*$ such that:
	\begin{equation}
	  \frac{1}{N}e' \frac{B(\sigma_1)}{2\alpha}\cdot \frac{B(\sigma_2)}{2\alpha}\cdot\cdots\cdot\frac{B(\sigma_k)}{2\alpha} \geq p^* e'.
		\label{expression3}
	\end{equation}
	Inequality (\ref{expression3}) implies that expression (\ref{expression2}), and consequently expression (\ref{expression}), is greater than or equal to $(2\alpha)^k p^* (t-\tau)^k/k!$. Now, if we write $1/N e' \Phi_B(t,\tau)$ as sum of expressions like (\ref{expression}), we have:
	\begin{equation}
	  1/N e' \Phi_B(t,\tau) \geq \sum_{k=0}^{\infty} \frac{(2\alpha)^kp^*(t-\tau)^k}{k!} = p^*e^{2\alpha(t-\tau)}.
	\end{equation}
	Thus,
	\begin{equation}
	  1/N e' \Phi_A(t,\tau) \geq p^*e^{2\alpha(t-\tau)}.e^{-2\alpha(t-\tau)} = p^*,
	\end{equation}
	and from Lemma \ref{DS-pstar} the theorem is proved.
 \end{proof}

%%%%%%%%%%%%%%%%%%%%%%%%%%%%%%%%%%%%%%%%%%%%%%%%%%%%%%%%%%%%%%%%%%%%
%%%%%%%%%%%%%%%%%%%%%%%%%%%%%%%%%%%%%%%%%%%%%%%%%%%%%%%%%%%%%%%%%%%%
%%%%%%%%%%%%%%%%%%%%%%%%%%%%%%%%%%%%%%%%%%%%%%%%%%%%%%%%%%%%%%%%%%%%

\section{Conclusion}
\label{Conclusion-DS}

We considered a general linear distributed averaging algorithm in both discrete time and continuous time. Following \cite{Touri:3}, and recalling the notion of jets from \cite{Blackwell:45}, we introduced a property of chains of stochastic matrices, more precisely, the infinite jet-flow property in the discrete time case. The latter property is shown to be a strong necessary condition for ergodicity of the chain. Moreover, for the chain to be class-ergodic, the infinite jet-flow property must hold over each connected component of the infinite flow graph, as defined in \cite{Touri:11c}.

We then illustrated the close relationship between Sonin's D-S Theorem and convergence properties of linear consensus algorithms. By employing the D-S Theorem, we showed in the discrete time case that the necessary conditions found earlier are also sufficient in case the chain is in Class $\mathcal{P}^*$ \cite{Touri:11c}. We argued that the obtained equivalent conditions for ergodicity and class-ergodicity of chains in Class $\mathcal{P}^*$ can subsume the previous related results in the literature, \cite{Hend:11,Bolouki:12b,Touri:11c} in particular.

A geometric approach was then introduced to interpret the jets in the D-S Theorem. The approach turned out to be a powerful method to rediscover our aforementioned results, and also to extend them to the continuous time case. In future work, we shall attempt an extension of our results to the case when the number of agents increases to infinity, although the D-S Theorem holds only if $N$ is finite.

%%%%%%%%%%%%%%%%%%%%%%%%%%%%%%%%%%%%%%%%%%%%%%%%%%%%%%%%%%%%%%%%%%%%
%%%%%%%%%%%%%%%%%%%%%%%%%%%%%%%%%%%%%%%%%%%%%%%%%%%%%%%%%%%%%%%%%%%%
%%%%%%%%%%%%%%%%%%%%%%%%%%%%%%%%%%%%%%%%%%%%%%%%%%%%%%%%%%%%%%%%%%%%

\bibliography{ExtendedVersion}
\bibliographystyle{IEEEtran}

%%%%%%%%%%%%%%%%%%%%%%%%%%%%%%%%%%%%%%%%%%%%%%%%%%%%%%%%%%%%%%%%%%%%
%%%%%%%%%%%%%%%%%%%%%%%%%%%%%%%%%%%%%%%%%%%%%%%%%%%%%%%%%%%%%%%%%%%%
%%%%%%%%%%%%%%%%%%%%%%%%%%%%%%%%%%%%%%%%%%%%%%%%%%%%%%%%%%%%%%%%%%%%

\end{document}